\documentclass[12pt]{article}
\usepackage{amsmath,amssymb,amsthm, amsfonts}
\usepackage{hyperref}
\usepackage{graphicx}
\usepackage{array, tabulary}
\usepackage{tabu}
\usepackage{url}
\usepackage[mathlines]{lineno}
\usepackage{dsfont} 
\usepackage{color}
\usepackage{subcaption}
\usepackage{enumitem}
\definecolor{red}{rgb}{1,0,0}

\definecolor{blue}{rgb}{0,0,1}

\definecolor{green}{rgb}{0,.6,0}

\usepackage{float}
\usepackage{bm}

\usepackage{tikz}

\newtheorem{thm}{Theorem}[section]
\newtheorem{cor}[thm]{Corollary}
\newtheorem{lem}[thm]{Lemma}
\newtheorem{prop}[thm]{Proposition}

\newtheorem{obs}[thm]{Observation}

\newtheorem{rmk}[thm]{Remark}

\theoremstyle{definition}

\theoremstyle{definition}
\newtheorem{defn}[thm]{Definition}

\theoremstyle{definition}
\newtheorem{ex}[thm]{Example}

\newcommand{\F}{\mathcal{F}}

\newcommand{\Z}{\operatorname{Z}}

\newcommand{\bit}{\begin{itemize}}
\newcommand{\eit}{\end{itemize}}
\newcommand{\ben}{\begin{enumerate}}
\newcommand{\een}{\end{enumerate}}
\newcommand{\beq}{\begin{equation}}
\newcommand{\eeq}{\end{equation}}
\newcommand{\bea}{\begin{eqnarray}} 
\newcommand{\eea}{\end{eqnarray}}
\newcommand{\bpf}{\begin{proof}}
\newcommand{\epf}{\end{proof}\ms}
\newcommand{\bmt}{\begin{bmatrix}}
\newcommand{\emt}{\end{bmatrix}}
\newcommand{\ms}{\medskip}

\newcommand{\noi}{\noindent}

\newcommand{\beqs}{\begin{equation*}} 
\newcommand{\eeqs}{\end{equation*}}
\newcommand{\beas}{\begin{eqnarray*}}
\newcommand{\eeas}{\end{eqnarray*}}

\newcommand{\up}[1]{^{(#1)}}

\newcommand{\calf}{\mathcal{F}}

\newcommand{\zf}{\operatorname{\lfloor \operatorname{Z} \rfloor}}
\newcommand{\Zf}{\zf}

\renewcommand{\H}{\operatorname{H}}

\newcommand{\ptz}{\operatorname{pt_{\Z}}}

\newcommand{\pth}{\operatorname{pt}_{\H}}

\newcommand{\thz}{\operatorname{th_{\Z}}}
\newcommand{\thzs}{\operatorname{th}_{\Z}^*}

\newcommand{\thzf}{\operatorname{th_{\zf}}}

\newcommand{\thh}{\operatorname{th}_{\H}}
\newcommand{\thhs}{\operatorname{th}_{\H}^*}
\newcommand{\thht}{\operatorname{th}_{\H}^\times}

\newcommand{\kz}{k_{\Z}}
\newcommand{\kh}{k_{\H}}
\newcommand{\X}{\operatorname{X}}
\newcommand{\thx}{\operatorname{th_{\X}}}
\newcommand{\ptx}{\operatorname{pt}_{\X}}

\newcommand{\srg}{\operatorname{SRG}}
\newcommand{\rev}{\operatorname{rev}}
\newcommand{\Rev}{\operatorname{Rev}}
\newcommand{\Term}{\operatorname{Term}}

\newcommand{\CCRZf}{\operatorname{CCR-\lfloor Z \rfloor}}

\newcommand{\cart}{\, \square \,}
\newcommand{\ol}{\overline}
\newcommand{\mc}{\mathcal}
\newcommand{\minor}{\preceq}
\newcommand{\force}{\rightarrow}
\newcommand{\iso}{\cong}
\renewcommand{\L}{\mc{L}}

\newcommand{\lceill}{\raisebox{2.5pt}{$\lceil$}}
\newcommand{\rceill}{\raisebox{2.5pt}{$\rceil$}}

\title{The Hopping Forcing Rule}
\author{Joshua Carlson \thanks{Department of Mathematics and Computer Science, Drake University, Des Moines, IA, USA (joshua.carlson@drake.edu)} \and John Petrucci \thanks{Department of Mathematics and Statistics, Williams College, Williamstown, MA, USA (john.petrucci507@gmail.com)}}
\date{\today}

\begin{document}

\maketitle

\begin{abstract}
Zero forcing is a combinatorial game played on graphs that can be used to model the spread of information with repeated applications of a color change rule. In general, a zero forcing parameter is the minimum number of initial blue vertices that are needed to eventually color every vertex blue with a given color change rule. Furthermore, the throttling number minimizes the sum of the number of initial blue vertices and the time taken for all vertices to become blue. In 2013, Barioli et al.~added a new rule, called hopping, to existing color change rules in order to demonstrate that the minor monotone floor of various zero forcing parameters is itself, a zero forcing parameter. In this paper, we examine the hopping color change rule independently from the other classic rules. Specifically, we study the hopping forcing number and the hopping throttling number. We investigate the ways in which these numbers are related to various graph theory parameters (such as vertex connectivity and independence number) as well as other zero forcing parameters.
\end{abstract}

\noi {\bf Keywords} Zero forcing, Color change rules, Hopping, Throttling

\noi{\bf AMS subject classification} 05C57, 05C15, 05C50
\section{Introduction}

Zero forcing was originally introduced in \cite{AIM08} as a tool for bounding the maximum nullity of a family of matrices associated with a given graph. One of the key features of zero forcing that makes it an effective tool is that it can be described as a simple game played on graphs whose vertices are colored blue and white. The goal of the game is to iteratively apply a color change rule to change the color of every vertex to blue. Note that such a game can also be used to model the spread of information on a graph where blue vertices have the information and white vertices do not. This motivates the study of zero forcing for its combinatorial properties and it is no surprise that many applications have been found in physics and engineering (see \cite{BH05, BG07,S08}). Since there are numerous variants of the classic zero forcing game, it is easiest to start with some abstract definitions that unify many of the important parameters. We follow most of the graph theoretical notation in \cite{Diestel} and all graphs in this paper are simple, finite, and undirected.

An \emph{abstract color change rule} is a set of conditions that specify when a vertex $v$ can force a white vertex $w$ to become blue. For example, the \emph{standard color change rule}, denoted $\Z$, is that a blue vertex $v$ can force a white vertex $w$ to become blue if $w$ is the unique white neighbor of $v$. Alternatively, the \emph{hopping color change rule}, denoted $\H$, is that a blue vertex $v$ can force a white vertex $w$ to become blue if $v$ has yet to perform a force and each vertex in $N(v)$ is blue. As we continue with the needed abstract definitions, it is useful to keep $\Z$ and $\H$ in mind as they provide some motivating comparisons.

Suppose $\X$ is an abstract color change rule and $G$ is a graph with initial blue set $B \subseteq V(G)$. As we repeatedly apply the color change rule, we can record each force in order (writing $v \overset{\X}{\rightarrow} w$ to indicate that $v$ forced $w$) until no more forces are possible. Note that we can drop the $\X$ in $v \overset{\X}{\rightarrow} w$ if the rule is clear from context. This ordered list is called a \emph{chronological list of forces of $B$} and the un-ordered set of forces in the list is called a \emph{set of forces of $B$}. An \emph{$\X$ forcing chain} of a set of forces $\calf$ is a maximal list of vertices $v_1, \ldots, v_k \in V(G)$ such that $(v_i \overset{\X}{\rightarrow} v_{i+1}) \in \calf$ for each integer $1 \leq i \leq k-1$. Additionally, the set of blue vertices in $V(G)$ after performing a chronological list of forces of $B$ is called an \emph{$\X$ final coloring of $B$}. If $V(G)$ is an $\X$ final coloring of $B$, then $B$ is an \emph{$\X$ forcing set of $G$}. The \emph{$\X$ forcing number of $G$}, denoted $\X(G)$, is the size of a minimum $\X$ forcing set of $G$. For standard zero forcing (and many others), chronological lists and sets of forces of a given subset $B \subseteq V(G)$ are not unique. However, it is important to note that if $\X = \Z$, then the $\X$ final coloring of $B$ is in fact unique. This property is not held by all color change rules (especially $\H$).

In \cite{HHK12}, the authors introduce the concept of propagation for standard zero forcing by performing multiple forces simultaneously in a sequence of time steps. This idea can be made abstract using sets of forces. For a subset $B \subseteq V(G)$ and a set of forces $\calf$ of $B$, define $\calf\up{0} = B$. Then, for each integer $t > 0$, $\calf \up{t}$ is the set of vertices $w \in V(G)$ for which there exists a vertex $v \in \bigcup_{i=0}^{t-1} \calf \up{i}$ such that $v \overset{\X}{\rightarrow} w$ is a valid force in $\calf$ given that $\bigcup_{i=0}^{t-1} \calf \up{i}$ is blue and $\overset{\phantom{}}{V(G)} \setminus \left(\bigcup_{i=0}^{t-1} \calf \up{i}\right)$ is white. In other words, for each $t >0$, $\calf \up{t}$ is the set of vertices that can become blue simultaneously in time step $t$ using $\calf$. The \emph{$\X$ propagation time of $\calf$}, denoted $\ptx(G;\calf)$, is the smallest integer $t$ such that $\bigcup_{i=0}^{t} \calf \up{i} = V(G)$. The \emph{$\X$ propagation time of $B$} is defined as $\ptx(G;B) = \min\{\ptx(G;\calf) \ | \ \calf \text{ is a set of $\X$ forces of }B\}$. Observe that if $B$ is not an $X$ forcing set of $G$, then no set of forces of $G$ colors $V(G)$ blue which means $\ptx(G;B) = \infty$. We can also optimize the balance between the size of an initial set $B \subseteq V(G)$ and its propagation time with the concept of throttling. The \emph{$\X$ throttling number of $G$} is defined as $\thx(G) = \min\{\thx(G;B) \ | \ B \subseteq V(G)\}$ where $\thx(G;B) = |B| + \ptx(G;B)$.

There are some important remarks to be made about the abstract definitions of propagation time. While the preceding definition of $\ptx(G;\calf)$ is perfectly analogous to $\ptz(G; \calf)$ in \cite{HHK12}, this is not the case for $\ptx(G;B)$ and $\ptz(G;B)$. This discrepancy stems from the difference between color change rules like $\Z$ and those like $\H$ mentioned above. In fact, for standard zero forcing, since final colorings are unique to the initial blue set $B$, we can obtain each $B \up{t}$ by starting with $B$ blue and performing all possible simultaneous forces during each time step. This is an attractive feature of many variants of zero forcing (see \cite{BBF13, CHK19, skew, HHK12}) and the need (in general) to minimize over numerous sets of forces to determine $\ptx(G;B)$ could be seen a deterrent for investigating more complex color change rules. One goal of this paper is to change that mindset by studying hopping forcing and the hopping throttling number.

The hopping color change rule was introduced in \cite{BBF13} as one part of another rule called $\CCRZf$. Specifically, the rule $\CCRZf$ is to either apply the standard color change rule or the hopping color change rule. The authors in \cite{BBF13} developed $\CCRZf$ as a way to show that the minor monotone floor of the zero forcing number, defined as $\zf(G) = \min\{\Z(G') \ | \ G' \text{ is a major of }G\}$, is in fact a zero forcing parameter.

\begin{thm}[{\cite[Theorem 2.39]{BBF13}}]\label{thm:ccrzf}
The parameter $\Zf$ is the $\CCRZf$ parameter. In other words, for every graph $G$, $\Zf(G)=\CCRZf(G)$.
\end{thm}

For this reason, we can abbreviate the rule $\CCRZf$ to $\zf$. In \cite{BBF13}, numerous other graph parameters are shown to be equivalent to a variant of the zero forcing number using some color change rule. Furthermore, the color change rule for many minor monotone floor parameters adapted an existing color change rule by adding hopping. However, hopping was never studied in isolation from the other color change rules. The second goal of this paper is to investigate the connections that hopping forcing has to other important graph parameters. First, we need to address some nuance of the hopping color change rule. In \cite{BBF13}, the terms \emph{active} and \emph{inactive} are used to indicate vertices that are or are not able to force. However, it is useful to further partition the inactive vertices. In the following definition, we borrow some language from volcanology and clarify these concepts.

\begin{defn}
Let $b$ be a blue vertex in $V(G)$ at a particular time $t$ in a hopping forcing process. Then at time $t$, $b$ is \emph{dormant} if it has not yet performed a force but is unable to, \emph{active} if it has not yet performed a force but is able to, or \emph{extinct} if it has already performed a force.
\end{defn}

In Section \ref{sec:hopforcenumber}, we establish some bounds on the hopping forcing number and determine its value for various graph families. Next, we study the hopping throttling number in Section \ref{sec:throttling}. Specifically, we obtain and investigate bounds for $\thh(G)$ in terms of the vertex connectivity and the independence number in Sections \ref{sec:throtlowerbounds} and \ref{sec:throtUpperBounds} respectively. We characterize the graphs with extreme hopping throttling numbers in Section \ref{sec:throtExtreme} and we compare $\thh(G)$ with other throttling numbers in Section \ref{sec:comparingThrottling}. Then, in Section \ref{sec:productThrottling}, we explore product throttling for hopping forcing. Finally, we discuss some directions for future research in Section \ref{sec:conclusion}. 

\section{The hopping forcing number}\label{sec:hopforcenumber}

In this section, to begin our investigation of hopping forcing, we explore the hopping forcing number and its connections to other forcing numbers. First, to start hopping on a graph $G$, we must have one blue vertex $v$ with all of its neighbors colored blue. At minimum, $v$ has $\delta(G)$ neighbors. So, $\H(G)\geq \delta(G)+1$. Further, $V(G)$ is always a hopping forcing set of $G$. Combined, these bounds give us the following observation.

\begin{obs}\label{obs:Hdelta}
For any graph $G$ of order $n$, $\delta(G)+1 \leq \H(G) \leq n$.
\end{obs}

Suppose $\H(G)=1$. Then, starting with one vertex colored blue, we must be able to force the vertices of $G$, necessarily one in each time step, until all of them are colored blue. This means that at no point can a vertex we force be adjacent to another vertex, since otherwise it would be dormant. This implies $G$ is isomorphic to $\ol{K_n}$. Noting that $\H(\ol{K_n})=1$, we have that $\H(G)$ takes on its minimum value of $1$ if and only if $G\iso\ol{K_n}$. We derive a similar result concerning the maximum value of $\H(G)$.

\begin{prop}\label{prop:hkn}
For any graph $G$ of order $n$, $\H(G)=n$ if and only if $G\iso K_n$.
\end{prop}
\begin{proof}
First, by Observation \ref{obs:Hdelta}, note $\H(K_n)\leq n$ and $\H(K_n)\geq\delta(G)+1=(n-1)+1=n$, so $\H(K_n)=n$. Assume $\H(G)=n$. Then, any set $B$ of $n-1$ vertices being blue cannot force the remaining vertex $v$ to become blue. If any vertex in $B$ was not adjacent to a white vertex, they would be able to force the last vertex. Thus, every vertex in $B$ is adjacent to $v$ and so $v$ is a universal vertex. Since this is true for any set $B$ of $n-1$ vertices, every vertex is universal, and so $G$ must be $K_n$.
\end{proof}

By making use of these elementary bounds, we can readily derive the hopping forcing number for several classic families of graphs (paths, cycles, wheels, and stars). Specifically, combining the lower bound based on the minimum degree from Observation \ref{obs:Hdelta} with the hopping forcing sets illustrated in Figure \ref{fig:easyHvalues}, we arrive at the following four observations.

\begin{obs}\label{obs:Hpath}
For all $n\geq2$, $\H(P_n)=2$.
\end{obs}
\begin{obs}\label{obs:Hcycle}
For all $n\geq3$, $\H(C_n)=3$.
\end{obs}
\begin{obs}\label{obs:Hwheel}
For all $n\geq4$, $\H(W_n)=4$.
\end{obs}
\begin{obs}\label{obs:Hstar}
For all $n\geq2$, $\H(K_{1,n-1})=2$.
\end{obs}

\begin{figure}[ht] \begin{center}
\scalebox{.3}{\includegraphics{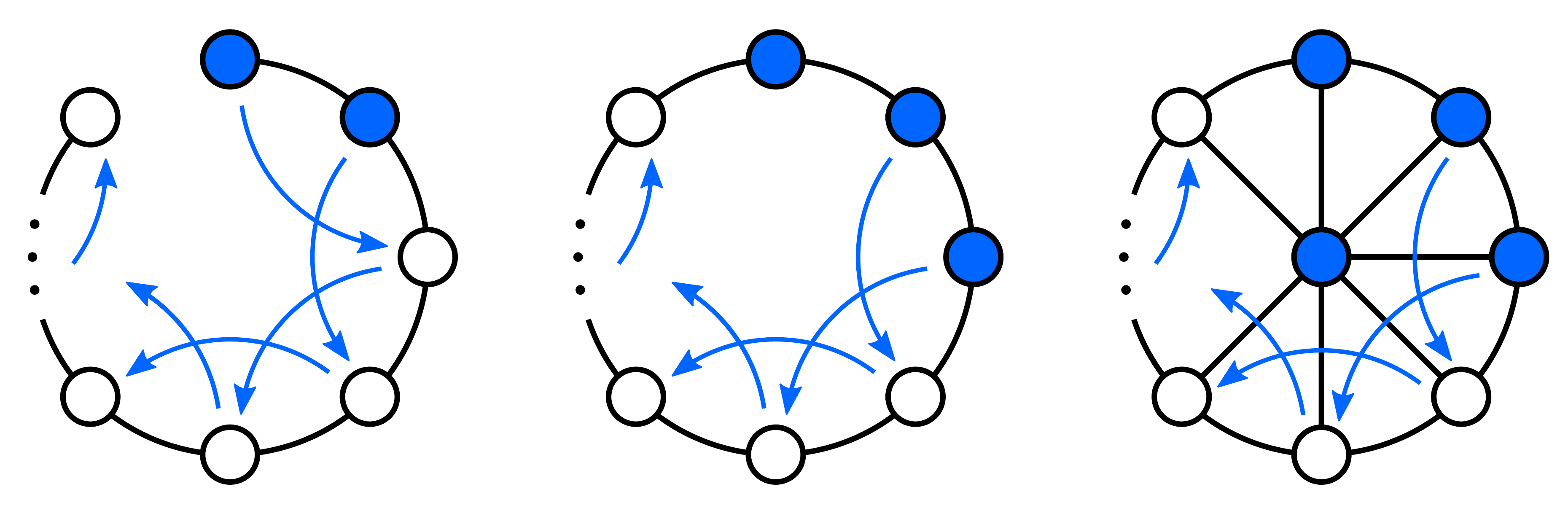}}\\
\caption{Minimum hopping forcing sets for $P_n$, $C_n$, $W_n$, and $K_{1,n-1}$.}
\label{fig:easyHvalues} 
\end{center}
\end{figure}

In fact, since $\Z(K_{1,n-1})=n-2$ (see \cite[Example 3.10]{C19}), Observation \ref{obs:Hstar} implies that the difference between $\Z(G)$ and $\H(G)$ can be made arbitrarily large. The remaining three observations above have $\H(G)=\Z(G)+1$. This is not a coincidence; the following proposition demonstrates how zero forcing sets can be turned into hopping forcing sets with the addition of one vertex, and further how hopping forcing sets double as $\Zf$ forcing sets.

\begin{prop}\label{prop:hineq}
For any graph $G$, $\Zf(G)\leq \H(G)\leq \Z(G)+1$.
\end{prop}
\begin{proof}
Since $\Zf$ forcing lets us use the hopping color change rule in addition to the standard color change rule, every hopping forcing set is also a $\Zf$ forcing set. Thus, $\Zf(G)\leq \H(G)$.

Let $B$ be a minimum zero forcing set of $G$, and let
\[
\mathcal{L}=( u_1\rightarrow v_1,\ u_2\rightarrow v_2,\ \ldots,\ u_n \rightarrow v_n )
\]
be a chronological list of standard zero forces of $B$. Define $B'=B\cup \{v_1 \}$; that is, $B'$ is the zero forcing set $B$ along with $v_1$, the first vertex that gets forced in $G$.

Starting with the vertices in $B'$ colored blue, we can force the entire graph blue by hopping as follows. To begin, $u_1$ forces $v_2$; then, $u_2$ forces $v_3$. In general, $u_i$ forces $v_{i+1}$ for each $1\leq i\leq n-1$.

Note that each of these forces follows the hopping color change rule. First, observe that in $\mathcal{L}$, $u_1$ forces $v_1$; thus, since $\mc{L}$ is a list of standard zero forces, $v_1$ is the only white neighbor of $u_1$ at time $0$. If we color $v_1$, then $u_1$ is surrounded by blue vertices and can force $v_2$ by a hop. To generalize, consider three consecutive forces 
\[
u_{i-1}\rightarrow v_{i-1},\ u_{i}\rightarrow v_{i},\ \text{and} \  u_{i+1}\rightarrow v_{i+1} 
\]
in $\mathcal{L}$ for some integer $2\leq i\leq n-1$. Since $u_i$ forces $v_i$ in $\mc{L}$, $v_i$ must be the only white neighbor of $u_i$ at the time immediately before the force occurs. Then, by having $u_{i-1}$ force $v_i$ by a hop, $u_i$ will be surrounded by blue vertices and can force $v_{i+1}$ by a hop.

This gives us a chronological list of hopping forces of $B'$, namely
\[
\mathcal{L}' = (u_1 \rightarrow v_2,\ u_2 \rightarrow v_3,\ \ldots, \  u_{n-1}\rightarrow v_n ).
\]
So, $B'$ is a hopping forcing set of $G$. Since $|B|=\Z(G)$, we have $|B'|=\Z(G)+1$. This implies $\H(G)\leq \Z(G)+1$, as desired.
\end{proof}

\begin{cor}\label{cor:hdelta}
If $\Z(G)=\delta(G)$, then $\H(G)=\delta(G)+1$.
\end{cor}
\begin{proof}
By Observation \ref{obs:Hdelta}, we have $\H(G)\geq\delta(G)+1$, and by Proposition \ref{prop:hineq}, we have $\H(G)\leq\Z(G)+1=\delta(G)+1$.
\end{proof}

It is easy to see that $\Z(P_n)=1$, $\Z(C_n)=2$, and $\Z(W_n)=3$; as such, Corollary \ref{cor:hdelta} serves as another way to determine the $\H$ values in Observations $\ref{obs:Hpath}$, $\ref{obs:Hcycle}$, and $\ref{obs:Hwheel}$. However, note that Corollary \ref{cor:hdelta} does not hold in the converse. To show this, consider the star graph $K_{1,n-1}$. Observation \ref{obs:Hstar} gives us that $\H(K_{1,n-1}) = 2 = \delta(K_{1,n-1})+1$, but as before, we have that $\Z(K_{1,n-1})=n-2\neq \delta(K_{1,n-1})$.

We conclude this section with a few more examples of the hopping forcing number for some famous graphs. First, we generalize Observation \ref{obs:Hstar} to consider all complete bipartite graphs.

\begin{prop}\label{prop:forcingKst}
Suppose $G$ is the complete bipartite graph $K_{s,t}$ on $s+t$ vertices with partite sets $U$ and $V$ such that $s=|U|\leq |V|=t$ with $s\geq 1$ and $t\geq 2$. Then $\Z(G)=s+t-2$ and $\H(G)=s+1$.
\end{prop}

\begin{proof}
Pick $u\in U$ and $v\in V$ and let $B=V(G)\setminus \{u,v \}$. Since each blue vertex in $U$ is adjacent to only one white vertex (namely $v$) and each blue vertex in $V$ is adjacent to only one white vertex (namely $u$), $u$ and $v$ can be forced in the first time step with the standard color change rule. This means $B$ is a standard zero forcing set of $G$ of size $|U|-1+|V|-1=s+t-2$, so $\Z(G)\leq s+t-2$.

Suppose that there exists a zero forcing set $B'$ of $G$ with size $|B'|<s+t-2$. Then, by the pigeonhole principle, there must be at least two white vertices in one of the partite sets of $G$; without loss of generality, suppose that partite set is $U$. At time $0$, each vertex in $V$ is then adjacent to at least two white vertices, and thus is dormant. Since there are no edges between vertices in $U$, no blue vertex in $U$ can force the white vertices of $U$ to become blue. So, $B'$ cannot be a zero forcing set. This implies that $\Z(G)=s+t-2$.

By Observation $\ref{obs:Hdelta}$, $\H(G)\geq s+1$. Let $B=U\cup \{v\}$. Then, $v$ is adjacent to only blue vertices (those in $U$), and as such is active and can force another vertex in $V$ to become blue. That vertex is now also active, and can force; this process continues until all vertices in $V$, and thus in $G$, are blue. As such, $B$ is a hopping forcing set, implying $\H(G)\leq s+1$ and $\H(G) = s+1$.
\end{proof}




We now find the hopping forcing number of the Petersen graph, illustrated in Figure \ref{fig:petersengraph}. Recall that a graph $G$ is \emph{strongly regular}, or $\srg(n, k, \lambda, \mu)$, if $|V(G)| = n$, $G$ is $k$-regular, every pair of adjacent vertices in $G$ has $\lambda$ common neighbors, and every pair of non-adjacent vertices in $G$ has $\mu$ common neighbors.

\begin{figure}[ht]
\begin{center}
\scalebox{.4}{\includegraphics{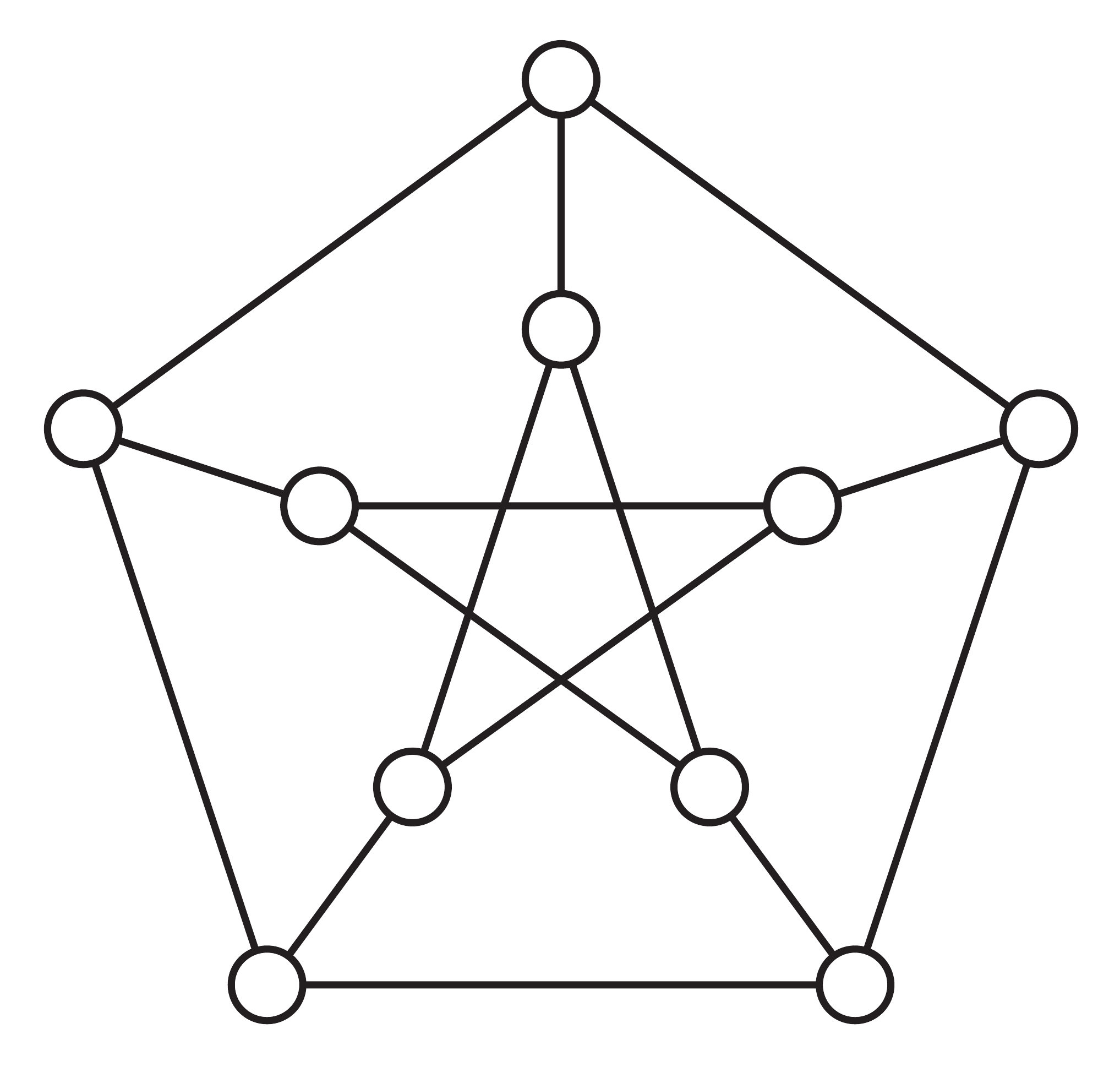}}\\
\caption{The Petersen graph.}
\label{fig:petersengraph} 
\end{center}
\end{figure}

\begin{prop}
If $P$ is the Petersen graph, then $\H(P)=6$.
\end{prop}
\begin{proof}
By \cite[Proposition 3.26]{AIM08}, $\Z(P)=5$. Thus, by Corollary \ref{cor:hdelta}, $\H(P)\leq 6$. Suppose there existed a hopping forcing set $B$ of $P$ with $|B|=5$. We need at least one active vertex in $B$ to be able to force at all, so $B$ must include some vertex $u$ along with its three neighbors $v_1$, $v_2$, and $v_3$. Since $P$ is $\srg(10,3,0,1)$, each pair of vertices in $\{v_1, v_2, v_3 \}$ is adjacent to only one common vertex, namely $u$, so they each must be adjacent to one blue vertex ($u$) and two white vertices. Then at best, we can color one of these vertices' white neighbors blue to complete $B$. Without loss of generality, color one white neighbor of $v_1$ blue, and call it $w$.

We can now begin forcing. The only active vertex at time $0$ is $u$. By the strong regularity of $P$, $w$ is adjacent to two white vertices, so the only blue vertex adjacent to exactly one white vertex is $v_1$. Let $x$ be the unique white neighbor of $v_1$. If $u$ does not force $x$ in the first time step, then we will have no active vertices at time $1$, so suppose $u$ forces $x$. This means that at time $1$, $v_1$ is the only active vertex; however, every dormant vertex is adjacent to two white vertices by the strong regularity of $P$. As such, there is no way for $v_1$ to force so that a dormant vertex becomes active, implying $B$ is not a hopping forcing set. This implies $\H(P)=6$.
\end{proof}


\section{Throttling for hopping forcing}\label{sec:throttling}

In this section, we turn to exploring throttling for hopping forcing, optimizing the trade-off between the size of a hopping forcing set and its propagation time. Specifically, we find several useful bounds on the hopping throttling number, examine its extreme values, and compare it to other types of throttling numbers.

\subsection{Lower bounds}\label{sec:throtlowerbounds}

In \cite{BY13}, Butler and Young give a lower bound for the standard throttling number, namely that $\thz(G)\geq \lceil 2\sqrt{n}-1 \rceil$ for any graph $G$ of order $n$. This bound relies on the fact that, under the standard color change rule, a blue vertex can force at most one other vertex to become blue in each time step. The same is true for hopping forcing---an active blue vertex cannot force more than one other vertex---so $\lceil 2\sqrt{n}-1 \rceil$ also bounds $\thh(G)$ from below.

However, we can improve on this bound by considering one key difference between hopping forcing and standard zero forcing. Under the standard color change rule, blue vertices adjacent to white vertices may or may not be able to force. However, under the hopping color change rule, we know that these vertices cannot force under any circumstances. We can then count the number of dormant vertices and, by extension, the number of active vertices at each time in a hopping forcing process by considering the minimum number of blue vertices adjacent to white vertices. This count is critically linked to the vertex connectivity of $G$, as demonstrated in the following theorem.

\begin{thm}\label{thm:thhkappa}
If $G$ is a graph on $n$ vertices with vertex connectivity $\kappa$, then
\[
\thh(G)\geq \lceil 2\sqrt{n-\kappa}+\kappa-1 \rceil.
\]
\end{thm}
\begin{proof}
Let $B$ be a hopping forcing set of $G$. Recall that a blue vertex $v$ in $V(G)$ must be either dormant, active, or extinct. In any case, during each time step, each blue vertex in $V(G)$ can force at most one white vertex to become blue, but never more.

If we start with $|B|$ blue vertices in $G$ and those vertices together perform $k$ forces during the first time step, then the number of blue vertices that have not yet forced (that is, the number of dormant or active vertices) becomes $|B|-k+k=|B|$, since active vertices can only force one white vertex to become blue. In general, at each time $t<\pth(G;B)$, the number of blue vertices that have not yet forced is equal to $|B|$.


For some $t<\pth(G;B)$, let $B'$ be the set of blue vertices that have not yet performed a force (all dormant or active vertices) at time $t$. Then, $|B'|=|B|$ by the logic above. Note that since $t<\pth(G;B)$, $V(G)$ is not all blue at time $t$. So, if we remove the dormant vertices at time $t$ from the graph, then no blue vertex is adjacent to a white vertex. This can only occur when the graph is disconnected, as otherwise there would be a dormant vertex we had not removed. By definition of vertex connectivity, the number of dormant vertices at time $t$ must be at least $\kappa$. Since $B'$ consists only of dormant and active vertices, and a blue vertex cannot be both dormant and active, $B'$ can be partitioned into dormant vertices and active vertices. So, there are at most $|B'|-\kappa = |B|-\kappa$ active vertices in $G$ at time $t$, which can force at most $|B|-\kappa$ white vertices to become blue in the succeeding time step $t+1$.



Thus, we start with $|B|$ blue vertices and in each time step, at most $|B|-\kappa$ white vertices become blue. After $\pth(G;B)$ time steps, all $n$ vertices of $G$ must be blue, so
\begin{eqnarray*}
|B| + \underbrace{(|B|-\kappa) + \ldots + (|B|-\kappa)}_{\pth(G;B) \text{ times}} \geq n &\iff& |B|+\pth(G;B)(|B|-\kappa) \geq n \\
&\iff & |B|+|B|\pth(G;B)-\kappa\pth(G;B) \geq n \\
&\iff & |B|\left(1+\pth(G;B) \right)-\kappa\pth(G;B) \geq n \\
&\iff & |B| \geq \frac{n+\kappa\pth(G;B)}{1+\pth(G;B)}.
\end{eqnarray*}
Let $b=|B|$ and $p=\pth(G;B)$. Then, to find $\thh(G)$, we want to minimize $b+p$ subject to $b\geq\frac{n+\kappa p}{1+p}$. The constraint gives $b+p \geq \frac{n+\kappa p}{1+p}+p$. Let $f(p)=\frac{n+\kappa p}{1+p}+p$. Then
\[
f'(p)=\frac{\kappa(1+p)-(n+\kappa p)(1)}{(1+p)^2}+1=\frac{\kappa-n}{(1+p)^2}+1.
\]
This means $f'(p)=0$ when $p=\pm\sqrt{n-\kappa}-1$; note that $\sqrt{n-\kappa}\geq 1$ since $\kappa \leq n-1$ for any graph. Since $p\geq0$, we obtain $p=\sqrt{n-\kappa}-1$ as our critical point. For $p>-1$, $f$ is concave-up, so $f(\sqrt{n-\kappa}-1)$ is the minimum. Thus
\begin{align*}
f(\sqrt{n-\kappa}-1) &= \frac{n+\kappa(\sqrt{n-\kappa}-1)}{1+\sqrt{n-\kappa}-1} + \sqrt{n-\kappa}-1 \\
&= \frac{n-\kappa+\kappa\sqrt{n-\kappa}}{\sqrt{n-\kappa}} + \sqrt{n-\kappa}-1 \\
&= \sqrt{n-\kappa} + \kappa + \sqrt{n-\kappa} - 1 \\
&= 2\sqrt{n-\kappa} + \kappa - 1.
\end{align*}
Since $\thh(G)$ is an integer, we have $\thh(G) = b+p\geq \lceil 2\sqrt{n-\kappa}+\kappa-1 \rceil$ as desired.
\end{proof}

Setting $\kappa=0$, we recover the original bound from \cite{BY13}.

\begin{cor}\label{cor:thhall}
For any graph $G$ of order $n$, $\thh(G)\geq \lceil 2\sqrt{n}-1 \rceil$.
\end{cor}
\begin{proof}
All graphs have $\kappa\geq0$, so by Theorem \ref{thm:thhkappa}, $\thh(G)\geq \lceil 2\sqrt{n-0}+0-1 \rceil = \lceil 2\sqrt{n}-1 \rceil$ as desired.
\end{proof}

However, we can get a slightly better bound when only considering connected graphs.

\begin{cor}\label{cor:thhconnected}
If $G$ is a connected graph on $n$ vertices, then $\thh(G)\geq \lceill 2\sqrt{n-1} \rceill$.
\end{cor}
\begin{proof}
A graph is connected if and only if it has vertex connectivity $\kappa\geq 1$. So, by Theorem \ref{thm:thhkappa}, $\thh(G)\geq \lceill 2\sqrt{n-1}+1-1 \rceill = \lceill 2\sqrt{n-1} \rceill$.
\end{proof}

The bound in Theorem \ref{thm:thhkappa} is very useful for determining the hopping throttling numbers for many different graph families. The first hopping throttling number we calculate is that of $\ol{K_n}$, the empty graph on $n$ vertices. By comparing $\ol{K_n}$ to the path graph $P_n$, we can use the standard throttling number of $P_n$ to readily derive the hopping throttling number of $\ol{K_n}$.

\begin{prop}\label{prop:thhempty}
For any empty graph $\ol{K_n}$, $\thh(\ol{K_n})=\lceil 2\sqrt{n}-1 \rceil$.
\end{prop}
\begin{proof}
By Corollary \ref{cor:thhall}, we have that $\thh(\ol{K_n})\geq \lceil 2\sqrt{n}-1 \rceil$, so it remains to show that $\thh(\ol{K_n})\leq \lceil 2\sqrt{n}-1 \rceil$.

Label the vertices of $\ol{K_n}$ as $1,2,\ldots,n$, and consider the path graph $P_n$ on these vertices with edges $\{ i,i+1 \}$ for $i\in\{1,\ldots,n-1\}$. Then, let $B$ be the zero forcing set that achieves the minimum value of $|B|+\ptz(P_n;B)$, and let $\mathcal{F}$ be the set of standard forces performed which realizes that throttling number. Since no two vertices in $\ol{K_n}$ are adjacent, any active blue vertex in $\ol{K_n}$ can force any white vertex via hopping, and so every standard force in $\mathcal{F}$ is also a valid hopping force on $\ol{K_n}$. As such, since $\thz(P_n)=\lceil 2\sqrt{n}-1 \rceil$ by \cite{BY13}, we have that $\thh(\ol{K_n})\leq \lceil 2\sqrt{n}-1 \rceil$. This gives $\thh(\ol{K_n})=\lceil 2\sqrt{n}-1 \rceil$ as desired.
\end{proof}

To continue, we determine $\thh(P_n)$. A common strategy used to find the throttling number of $P_n$ is to snake the path into a box, then write the size of a hopping forcing set and its propagation time in terms of the box's dimensions. This technique was first used in \cite{BY13} to show that $\thz(P_n)=\lceil 2\sqrt{n}-1 \rceil$, and later in \cite{CHK19} to find the value of $\thz(C_n)$. Here, we use the snaking strategy to show that $\thh(P_n)$ achieves the lower bound in Theorem \ref{thm:thhkappa} for $\kappa=1$.

\begin{prop}\label{prop:thhpath}
For any path $P_n$, $\thh(P_n)=\lceill 2\sqrt{n-1} \rceill$. 
\end{prop}
\begin{proof}
By Corollary \ref{cor:thhconnected}, we have that $\thh(P_n)\geq\lceill 2\sqrt{n-1} \rceill$, so it remains to show that $\thh(P_n)\leq \lceill 2\sqrt{n-1} \rceill$.

Let $m$ be the largest integer such that $m^2\leq n-1$. Then, we can embed $P_n$ into a box of height $m$ by snaking the path back and forth, with one leaf extending out of the box in the leftmost column. For example, Figure \ref{fig:p15snaking} illustrates the case where $n=15$ and $m=3$.

\begin{figure}[ht] \begin{center}
\scalebox{.4}{\includegraphics{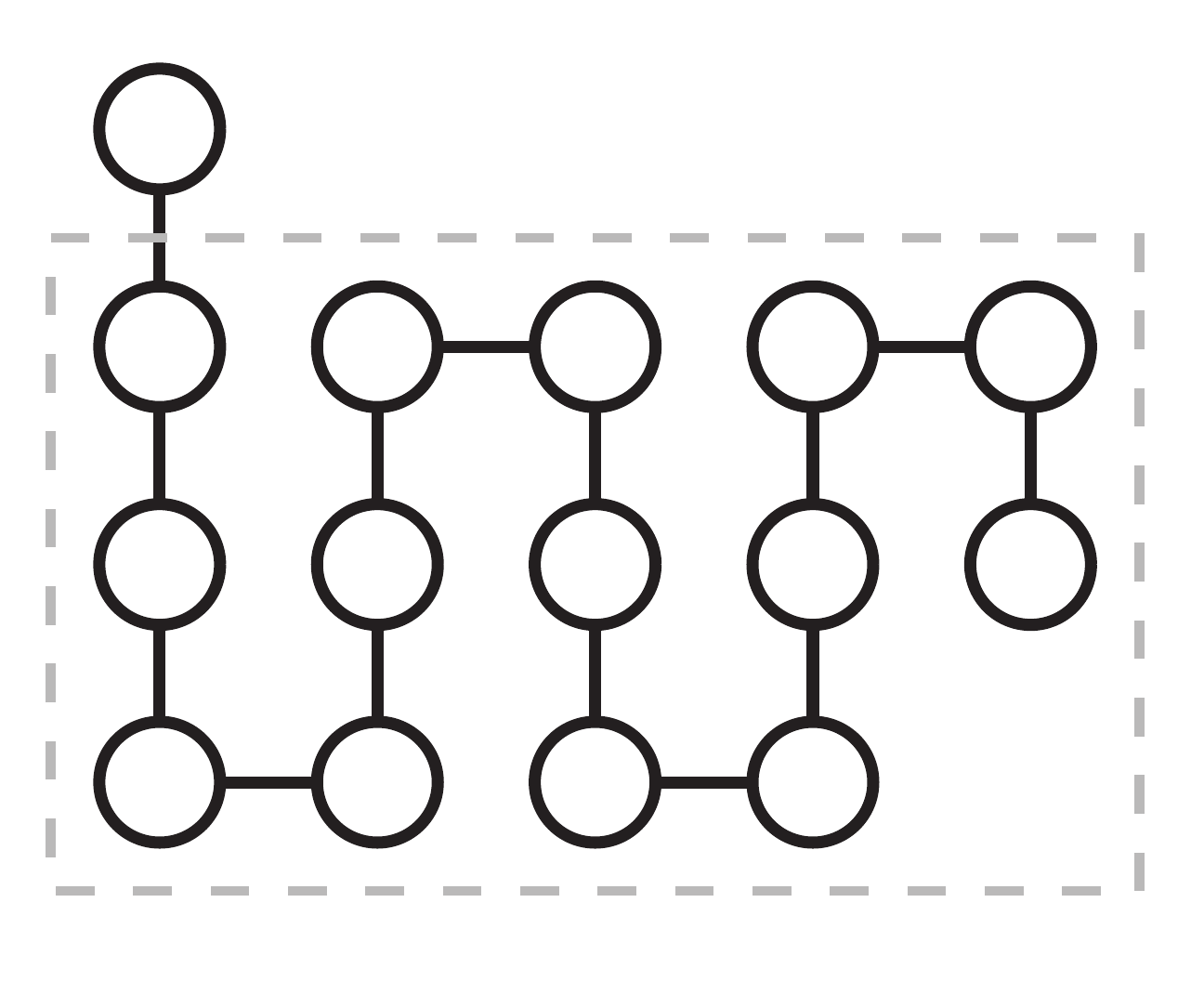}}\\
\caption{The snaking method for $P_{15}$, where $m=3$.}\label{fig:p15snaking} 
\end{center}
\end{figure}

Let $r=n-1-m^2$. If $r=0$, then $m^2=n-1$ and the path minus a leaf fits into an $m \times m$ box. If $1\leq r\leq m$, then we need an $m \times (m+1)$ box; lastly, if $m+1\leq r \leq 2m$, then we need an $m \times (m+2)$ box. Note that if $r<0$ or $r\geq2m+1$, then there exists a different largest integer $m_1$ for which $m_1^2\leq n-1$.

We then force as follows. Color the entire left column blue. In the first time step, each vertex in the leftmost column besides the bottom one forces the vertex positioned immediately down and to the right of it. Then, in the second time step, the bottom vertex in the first column forces the bottom vertex in the third column, and each vertex in the second column besides the topmost forces the vertex positioned immediately up and to the right of it. In the third time step, the top vertex in the second column forces the top vertex in the fourth column, and each vertex in the third column besides the bottommost forces the vertex positioned immediately down and to the right of it. We then alternate between the two processes from the second and third time steps, until every column has been colored blue. Figure \ref{fig:p15forcing} illustrates the forcing process for $P_{15}$, with the forcing set in blue and the numbers next to the forces representing the time step in which that force~occurs.

\begin{figure}[ht] \begin{center}
\scalebox{.4}{\includegraphics{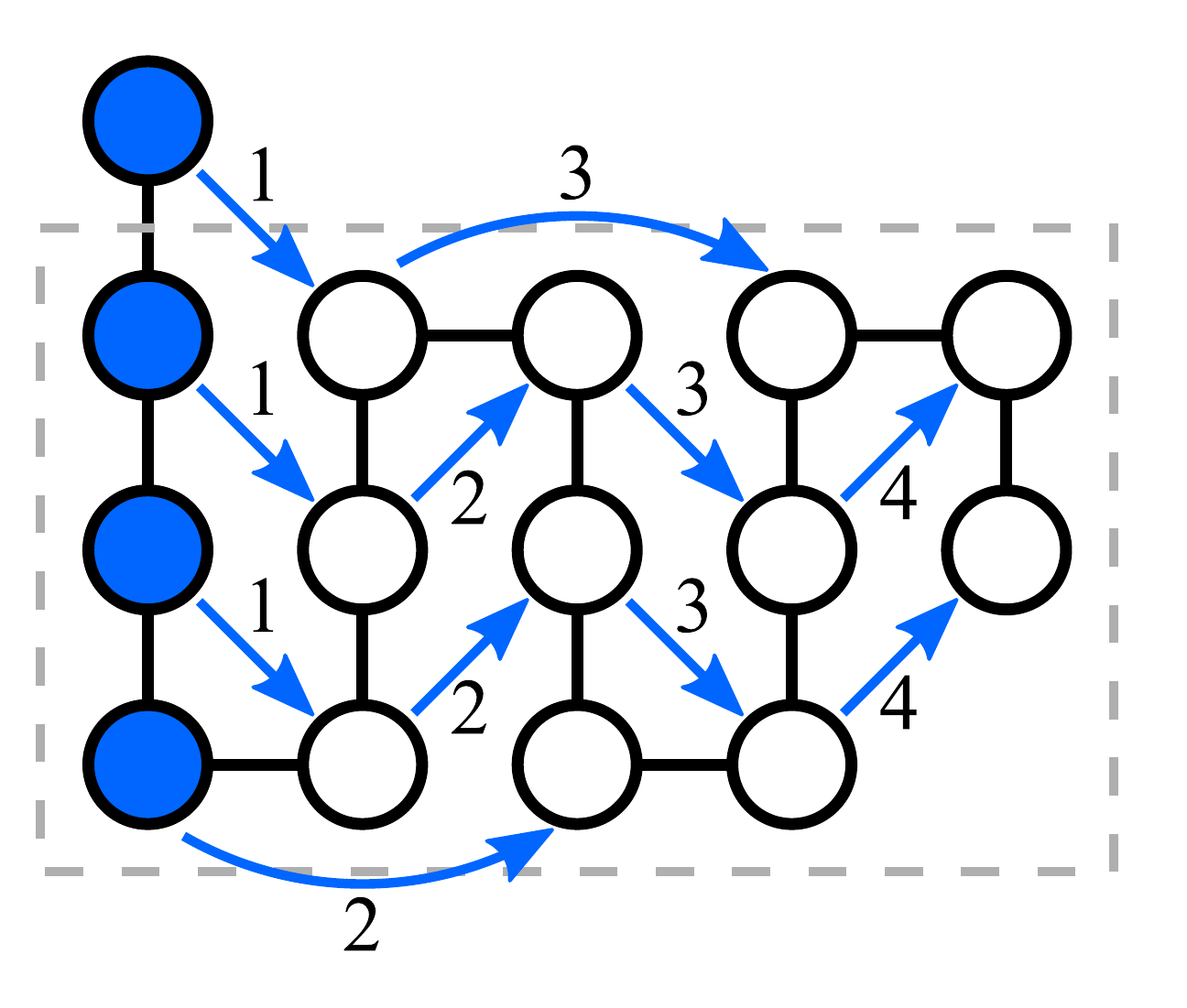}}\\
\caption{The forcing process on $P_{15}$.}\label{fig:p15forcing} 
\end{center}
\end{figure}

The forcing process described above colors one column in each time step. Thus, since we start with the $m+1$ vertices in the first column colored blue, and since the propagation time is one fewer than the number of columns of the box, we have
\[
\thh(P_n)\leq\begin{cases}
2m &\text{if } r=0; \\
2m+1 &\text{if } 1\leq r\leq m; \\
2m+2 &\text{if } m+1\leq r\leq 2m.
\end{cases}
\]

It remains to show that in each case above, $\thh(P_n)\leq\lceill 2\sqrt{n-1} \rceill$. If $r=0$, then $m^2=n-1$, so $2m=2\sqrt{n-1}=\lceill 2\sqrt{n-1} \rceill$ since $n-1$ is a perfect square. If $1\leq r\leq m$, then $m^2+r=n-1$. We have
\[
2\sqrt{n-1} = 2\sqrt{m^2+r} \leq 2\sqrt{m^2+m} < 2\sqrt{m^2+m+\frac14} = 2\left(m+\frac12 \right) = 2m+1
\]
and
\[
2\sqrt{n-1} = 2\sqrt{m^2+r}\geq 2\sqrt{m^2+1}> 2\sqrt{m^2} = 2m.
\]
Thus, $2m<2\sqrt{n-1}< 2m+1$ and so $\lceill 2\sqrt{n-1} \rceill = 2m+1$. Finally, if $m+1\leq r\leq 2m$, then $m^2+r=n-1$. We have
\[
2\sqrt{n-1} = 2\sqrt{m^2+r}\leq 2\sqrt{m^2 + 2m} < 2\sqrt{m^2 + 2m + 1} = 2(m+1) = 2m+2
\]
and
\[
2\sqrt{n-1} = 2\sqrt{m^2+r} \geq 2\sqrt{m^2+m+1} > 2\sqrt{m^2+m+\frac14} = 2\left( m + \frac12 \right) = 2m+1.
\]
Then, $2m+1<2\sqrt{n-1}< 2m+2$ and so $\lceill 2\sqrt{n-1} \rceill = 2m+2$.

Altogether, this gives $\thh(P_n)\leq \lceill 2\sqrt{n-1} \rceill$, so $\thh(P_n) = \lceill 2\sqrt{n-1} \rceill$ as desired.
\end{proof}

A similar strategy, but with two vertices extending out of the box instead of one, can be used to show that the hopping throttling number of cycles achieves the lower bound for $\kappa=2$.

\begin{prop}\label{prop:thhcycle}
For any cycle $C_n$, $\thh(C_n)=\lceil 2 \sqrt{n-2}+1 \rceil$.
\end{prop}
\begin{proof}
Since the vertex connectivity of a cycle is $\kappa=2$, by Theorem \ref{thm:thhkappa}, we have that $\thh(C_n)\geq \lceil 2\sqrt{n-2} + 1 \rceil$.

Let $m$ be the largest integer such that $m^2\leq n-1$. Then, similar to the snaking method for paths outlined above, we can embed $C_n$ into a box of height $m$, with two vertices extending out of the box in the leftmost column.

Let $r=n-2-m^2$. If $r=0$, then $m^2=n-2$ and the cycle minus two vertices fits into an $m \times m$ box. If $1\leq r\leq m$, then we need an $m \times (m+1)$ box; lastly, if $m+1\leq r \leq 2m$, then we need an $m \times (m+2)$ box. Note that if $r<0$ or $r\geq2m+1$, then there exists a different largest integer $m_1$ for which $m_1^2\leq n-2$.

Forcing then occurs similarly to paths, but the top-left vertex never forces. For example, Figure \ref{fig:c16forcing} illustrates the case where $n=16$ and $m=3$.

\begin{figure}[ht] \begin{center}
\scalebox{.4}{\includegraphics{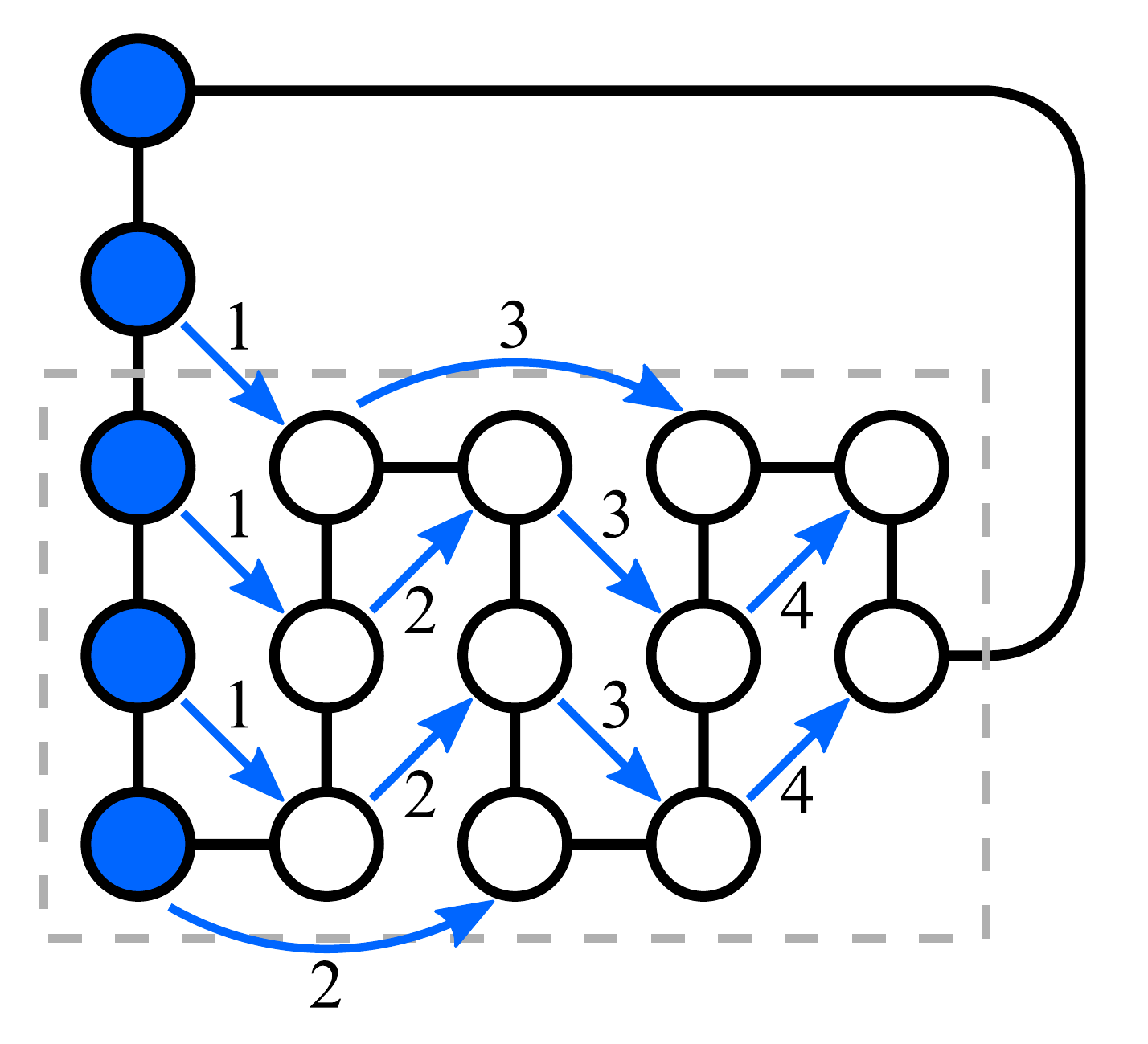}}\\
\caption{The snaking method for $C_{16}$, where $m=3$.}\label{fig:c16forcing} 
\end{center}
\end{figure}

The forcing process above colors one column in each time step. Thus, since we start with the $m+2$ vertices in the first column colored blue, and since the propagation time is one fewer than the number of columns of the box, we have
\[
\thh(C_n)\leq\begin{cases}
2m+1 &\text{if } r=0; \\
2m+2 &\text{if } 1\leq r\leq m; \\
2m+3 &\text{if } m+1\leq r\leq 2m.
\end{cases}
\]

It remains to show that in each case above, $\thh(C_n)\leq \lceil 2\sqrt{n-2}+1 \rceil$. If $r=0$, then $m^2=n-2$, so $2m+1=2\sqrt{n-2}+1=\lceil 2\sqrt{n-2}+1 \rceil$ since $n-2$ is a perfect square. If $1\leq r\leq m$, then $m^2+r=n-2$. We have
\[
2\sqrt{m^2+r}+1 \leq 2\sqrt{m^2+m} +1\leq 2\sqrt{m^2+m+\frac14}+1 = 2\left(m+\frac12 \right)+1 = 2m+2
\]
and
\[
2\sqrt{m^2+r}+1\geq 2\sqrt{m^2+1}+1> 2\sqrt{m^2}+1 = 2m+1.
\]
Thus, $2m+1<2\sqrt{n-1}\leq 2m+2$ and so $\lceil 2\sqrt{n-2}+1 \rceil = 2m+2$. Finally, if $m+1\leq r\leq 2m$, then $m^2+r=n-2$. We have
\[
2\sqrt{m^2+r}+1\leq 2\sqrt{m^2 + 2m} +1\leq 2\sqrt{m^2 + 2m + 1}+1 = 2(m+1)+1 = 2m+3
\]
and
\[
2\sqrt{m^2+r}+1 \geq 2\sqrt{m^2+m+1}+1 > 2\sqrt{m^2+m+\frac14} +1= 2\left( m + \frac12 \right)+1 = 2m+2.
\]
Then, $2m+2<2\sqrt{n-2}+1\leq 2m+3$ and so $\lceil 2\sqrt{n-2} +1 \rceil = 2m+3$. Altogether, this gives $\thh(C_n)\leq\lceil 2\sqrt{n-2} +1\rceil$, so $\thh(C_n) = \lceil 2\sqrt{n-2}+1 \rceil$ as desired.
\end{proof}

So far, we have shown that the lower bound in Theorem \ref{thm:thhkappa} is tight when $\kappa=1$ (paths) and when $\kappa=2$ (cycles). In fact, the complete bipartite graphs $K_{s,t}$ demonstrate that this bound is tight for all $\kappa$. By thinking of one partite set as an empty graph, we can use the throttling number of $\ol{K_n}$ to give a hopping strategy that meets the throttling lower bound.

\begin{prop}\label{prop:thhkappatight}
The lower bound in Theorem \ref{thm:thhkappa} is tight for all $\kappa\geq 0$.
\end{prop}
\begin{proof}
Consider the complete bipartite graph $K_{s,t}$ on $s+t$ vertices for any $s\geq 0$ and $t\geq 1$; assume $s\leq t$ with $s\geq 0$ and $t\geq 1$. Note that when $s=0$, $K_{0,t}=\ol{K_t}$. Observe, then, that $\kappa(K_{s,t})=\min\{s,t\}=s$, so by Theorem \ref{thm:thhkappa},
\[
\thh(K_{s,t}) \geq \lceil 2\sqrt{s+t-s}+s-1 \rceil = \lceil 2\sqrt{t} + s - 1 \rceil.
\]

To show the other inequality, we provide an explicit hopping forcing strategy that realizes this lower bound. Let $U$ be the partite set of $K_{s,t}$ of size $s$, and $V$ the partite set of size $t$. Observe that if we color $U$ blue, then coloring any vertex in $V$ blue makes it active, since each vertex in $V$ is only adjacent to vertices in $U$. Further, since $G[V]$ is isomorphic to $\ol{K_t}$, we are able to force $V$ to become blue using the same strategy as with $\ol{K_t}$ from Proposition \ref{prop:thhempty}.

So, consider the hopping forcing set consisting of $U$ and the subset of vertices in $V$ needed to optimize throttling on $V$. We know that the sum of the number of blue vertices in $V$ and the propagation time of the whole graph is $\lceil 2\sqrt{t}-1 \rceil$ by Proposition \ref{prop:thhempty}, so adding in the $s$ blue vertices from $U$, we have that
\[
\thh(K_{s,t})\leq s+ \lceil 2\sqrt{t}+1 \rceil = \lceil 2\sqrt{t} + s - 1 \rceil.
\]
This gives $\thh(K_{s,t}) = \lceil 2\sqrt{t} + s - 1 \rceil$, as desired.
\end{proof}

Figure \ref{fig:k35forcing} illustrates the hopping forcing process on $K_{3,5}$. By coloring every vertex in $U$ blue, we can then optimize throttling on $V$ to achieve the throttling number of $\thh(K_{3,5})=7$.

\begin{figure}[ht] \begin{center}
\scalebox{.4}{\includegraphics{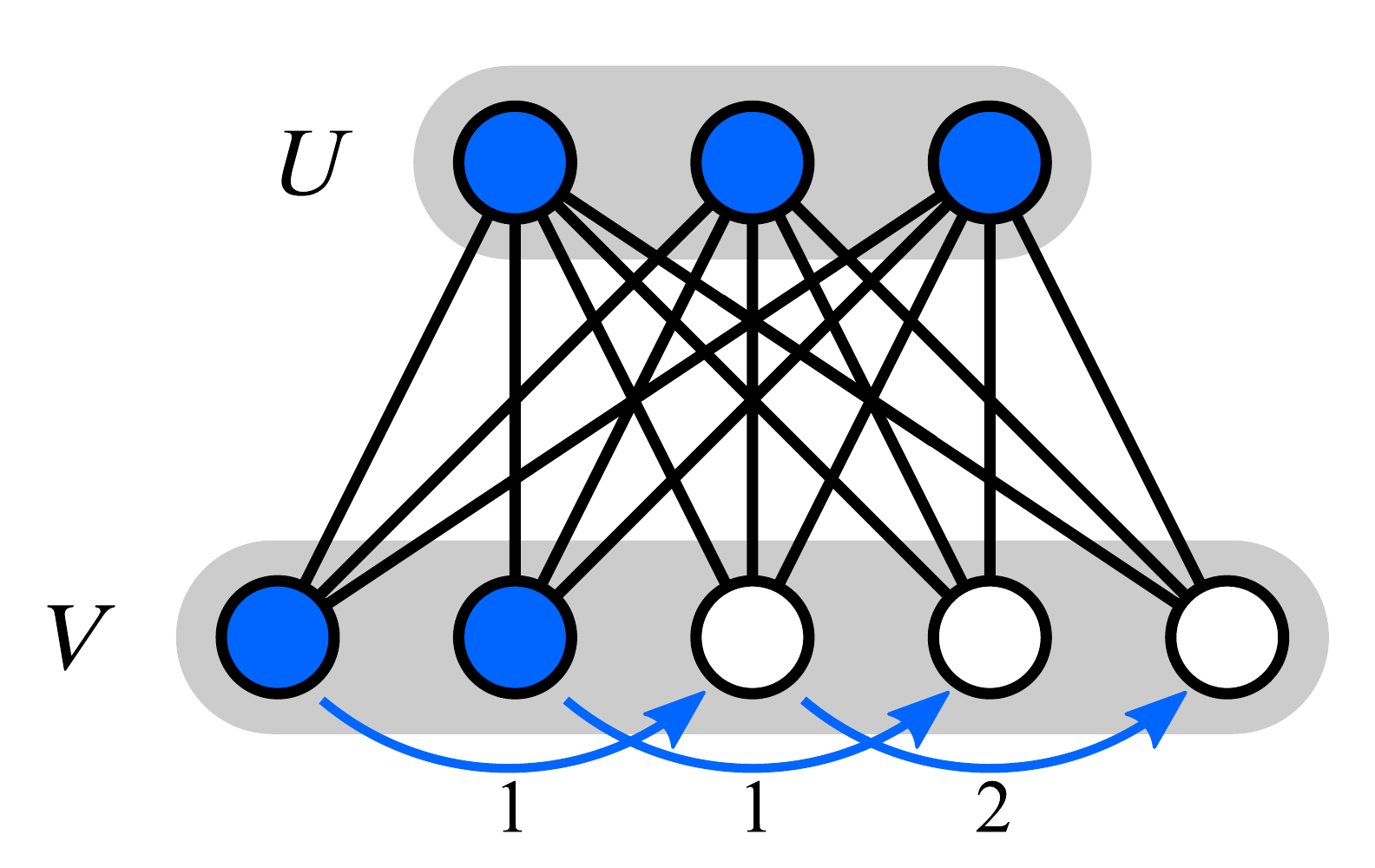}}\\
\caption{The forcing process on $K_{3,5}$.}
\label{fig:k35forcing} 
\end{center}
\end{figure}

We now know that infinitely many graphs realize the lower bound in Theorem \ref{thm:thhkappa} for every $\kappa\geq 0$. Then, it is natural to consider which other graphs might, and which might not, realize this lower bound. For example, so far, we have shown that two families of trees realize the lower bound---paths in Proposition \ref{prop:thhpath} and stars ($K_{1,n-1}$) in Proposition \ref{prop:thhkappatight}. Do all trees have $\thh(T)=\lceill 2\sqrt{n-1} \rceill$?

The answer is no. To show that there are infinitely many trees $T$ with $\thh(T)>\lceill 2\sqrt{n-1} \rceill$, we introduce the following family of graphs. Recall that a \emph{spider graph} is a tree with exactly one vertex of degree at least $3$, and all other vertices of degree at most $2$. A \emph{degree-3 spider} is a spider graph with maximum degree $3$; that is, it consists of a \emph{center} vertex adjacent to three disjoint paths (called \emph{legs}) of length $a$, $b$, and $c$, such as in Figure \ref{fig:deg3spider}. We will denote such graphs $S(a,b,c)$. Note that $S(a,b,c)$ has order $n=a+b+c+1$. 


\begin{figure}[ht] \begin{center}
\scalebox{.3}{\includegraphics{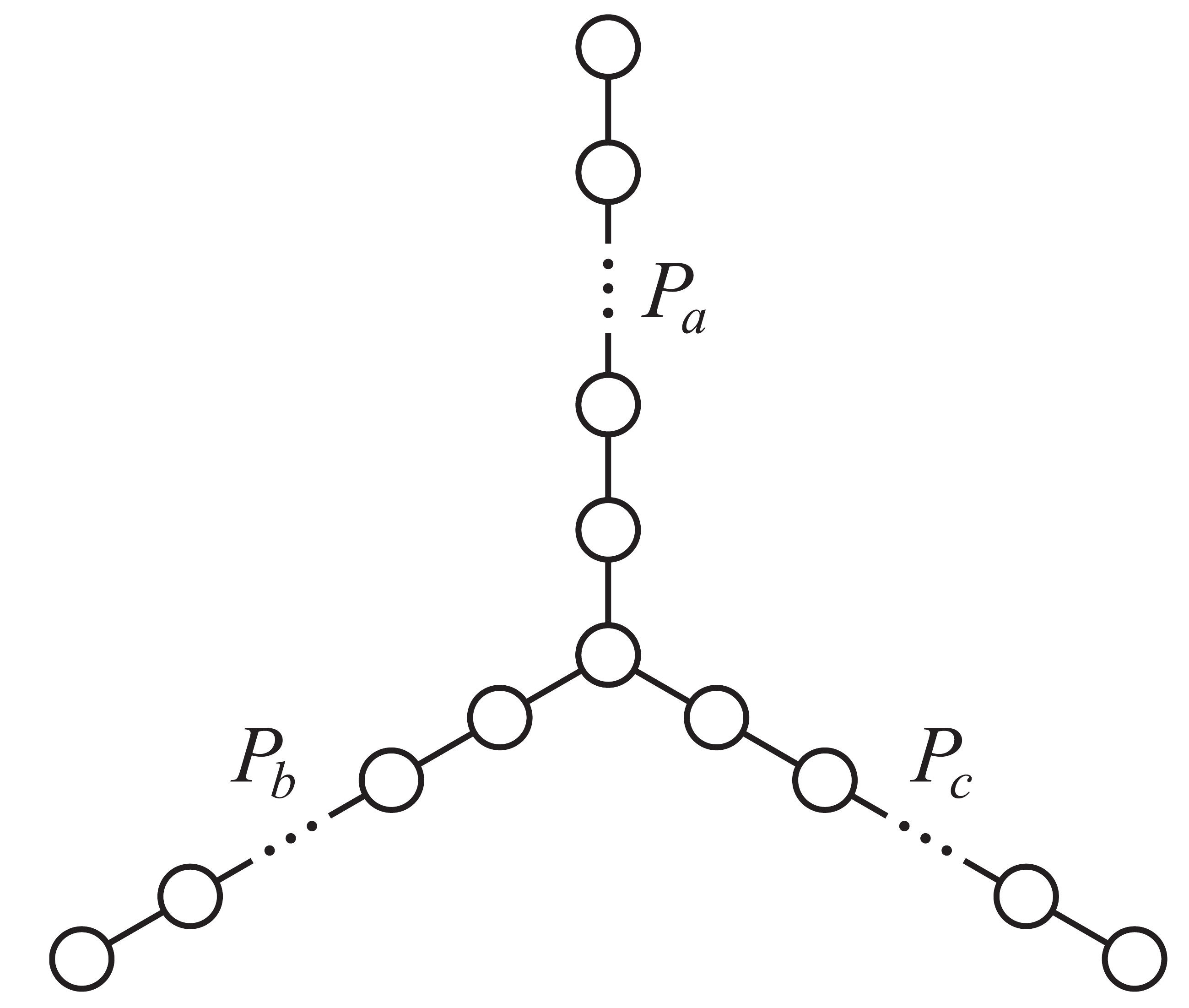}}\\
\caption{The degree-3 spider $S(a,b,c)$.}
\label{fig:deg3spider} 
\end{center}
\end{figure}

\begin{prop}\label{prop:thhtrees}
There exist infinitely many trees $T$ for which
\[
\thh(T) > \lceill 2\sqrt{n-\kappa}+\kappa-1 \rceill = \lceill 2\sqrt{n-1} \rceill.
\]
\end{prop}

\begin{proof}
Pick any integer $m\geq 2$, and consider the degree-3 spider 
\[
G=S(3m^2-1, 3m^2, 3m^2+1)
\]
on $9m^2+1$ vertices. Since $G$ is connected and removing the center vertex disconnects the graph, $\kappa(G)=1$. So, by Theorem \ref{thm:thhkappa},
\[
\thh(G)\geq \lceill 2\sqrt{9m^2+1-1} + 1 - 1 \rceill = 6m.
\]
We will show by contradiction that no hopping forcing set of $G$ realizes this lower bound.

Suppose that such a hopping forcing set $B$ does exist with $\thh(G;B)=6m$. Let $k=|B|$. Since we start with $k$ vertices colored blue, we must have that $\pth(G;B)=6m-k$ in order to realize this throttling number. Because $\kappa(G)=1$, the highest number of vertices that can be forced in any particular time step is $k-1$. As such, the maximum number of vertices that can be colored blue at the end of $6m-k$ time steps is $k+(6m-k)(k-1)$. Since $B$ is a hopping forcing set, all $9m^2+1$ vertices of $G$ must be colored blue by the end of the forcing process, so we need 
\begin{alignat*}{2}
        &\qquad & k+(6m-k)(k-1)                 &\geq 9m^2 + 1      \\
\iff    &       & k - k^2 + (6m+1)k-6m          &\geq 9m^2 + 1      \\
\iff    &       & -k^2 + (6m+2)k - (9m^2+6m+1)  &\geq 0             \\
\iff    &       & -k^2 + 2(3m+1)k - (3m+1)^2    &\geq 0             \\
\iff    &       & -(k-(3m+1))^2                 &\geq 0             \\
\iff    &       & 0                             &\geq (k-(3m+1))^2.
\end{alignat*}
The only solution to this inequality is $k=3m+1$. When $k=3m+1$, the maximum number of blue vertices after $6m-k=3m-1$ time steps is
\[
k+(6m-k)(k-1) = 3m+1+(6m-3m-1)(3m+1-1) = 9m^2+1 = |G|.
\]
As a result, in order for $B$ to realize the minimum throttling number, we must force $k-1$ vertices to become blue in every time step. In other words, we need to have only one dormant vertex at every time during the forcing process.

Consider the vertices in $B$, the blue vertices at time $0$. Note that the smallest of the legs in $G$ has $3m^2-1$ vertices, but for $m\geq 2$,
\begin{equation}\label{eqn:spiderineq}
3m^2-1>k=3m+1.
\end{equation}
Thus, in order to have only one dormant vertex in $B$, we must color $3m+1$ consecutive vertices at the end of one leg of $G$; any other set of size $3m+1$ will include at least two blue vertices each adjacent to a white vertex.

Afterwards, to continue to have only one dormant vertex at each time, we must force along the leg as we would with a path graph. Since we start with $3m+1$ vertices colored blue, and force $3m$ more to be blue in each time step, observe that after $m-1$ time steps,
\[
3m+1+(3m)(m-1)=3m^2+1
\]
vertices in $G$ are colored blue.

If $B$ consists of the $3m+1$ vertices at the end of the shortest leg of $G$ (that of length $3m^2-1$), then after $m-1$ time steps we will have forced that entire leg, plus the center and one additional vertex; the latter two must both be dormant.

If $B$ consists of the $3m+1$ vertices at the end of the longest leg of $G$ (that of length $3m^2+1$), then after $m-1$ time steps that entire leg will be blue. In the next time step, we have to force the center to avoid having two dormant vertices, plus $3m-1$ additional vertices. The only way to do this without having at least two dormant vertices would be to force another leg to be entirely blue, but by Equation \ref{eqn:spiderineq}, this is impossible; as such, we will end up with two dormant vertices after time step $m$.

Thus, our only option is for $B$ to consist of the $3m+1$ vertices at the end of the middle leg of $G$, which has length $3m^2$. After $m-1$ time steps, we will have forced that entire leg to be blue, plus the center, the only dormant vertex at this time step. However, in time step $m$, there is no way to force such that only one vertex is dormant. If we force the two white vertices adjacent to the center, we will end up with a dormant vertex on each path. Alternatively, if we try and force along a leg, the center and one vertex on that leg will both be dormant, since both remaining legs are, by Equation \ref{eqn:spiderineq}, too long to force entirely.

So, for any integer $m\geq 2$, there exists a time step in which we will not be able to force the maximum $k-1$ vertices. This implies the total number of blue vertices after $3m-1$ time steps is less than $n$, and so $\thh(G;B)>6m$ as desired.
\end{proof}

Proposition \ref{prop:thhtrees} gives an infinite family of graphs that do not realize the lower bound in Theorem \ref{thm:thhkappa} when $\kappa=1$. We can extend this result to all $\kappa\geq 0$ by considering a family in which we strategically add edges to a complete bipartite graph in order to keep the graph's vertex connectivity low, but increase its hopping throttling number.

\begin{prop}
For all $\kappa\geq 0$, there exist infinitely many graphs $G$ for which
\[
\thh(G) > \lceil 2\sqrt{n-\kappa}+\kappa-1 \rceil.
\]
\end{prop}
\begin{proof}
Define the graph $K(s,t) = \ol{K_s} \vee (K_{t-1} \cup K_1)$ for $s\leq t$. In other words, start with the complete bipartite graph $K_{s,t}$ with $0\leq s\leq t$ and $t\geq 1$, where $U$ is the partite set of size $s$ and $V$ the partite set of size $t$. Pick one vertex $v\in V$, then draw all possible edges between the vertices in $V'=V\setminus \{ v \}$ so that $V'$ induces $K_{t-1}$. The instance where $s=3$ and $t=5$ is illustrated in Figure \ref{fig:K35}.

\begin{figure}[ht] \begin{center}
\scalebox{.4}{\includegraphics{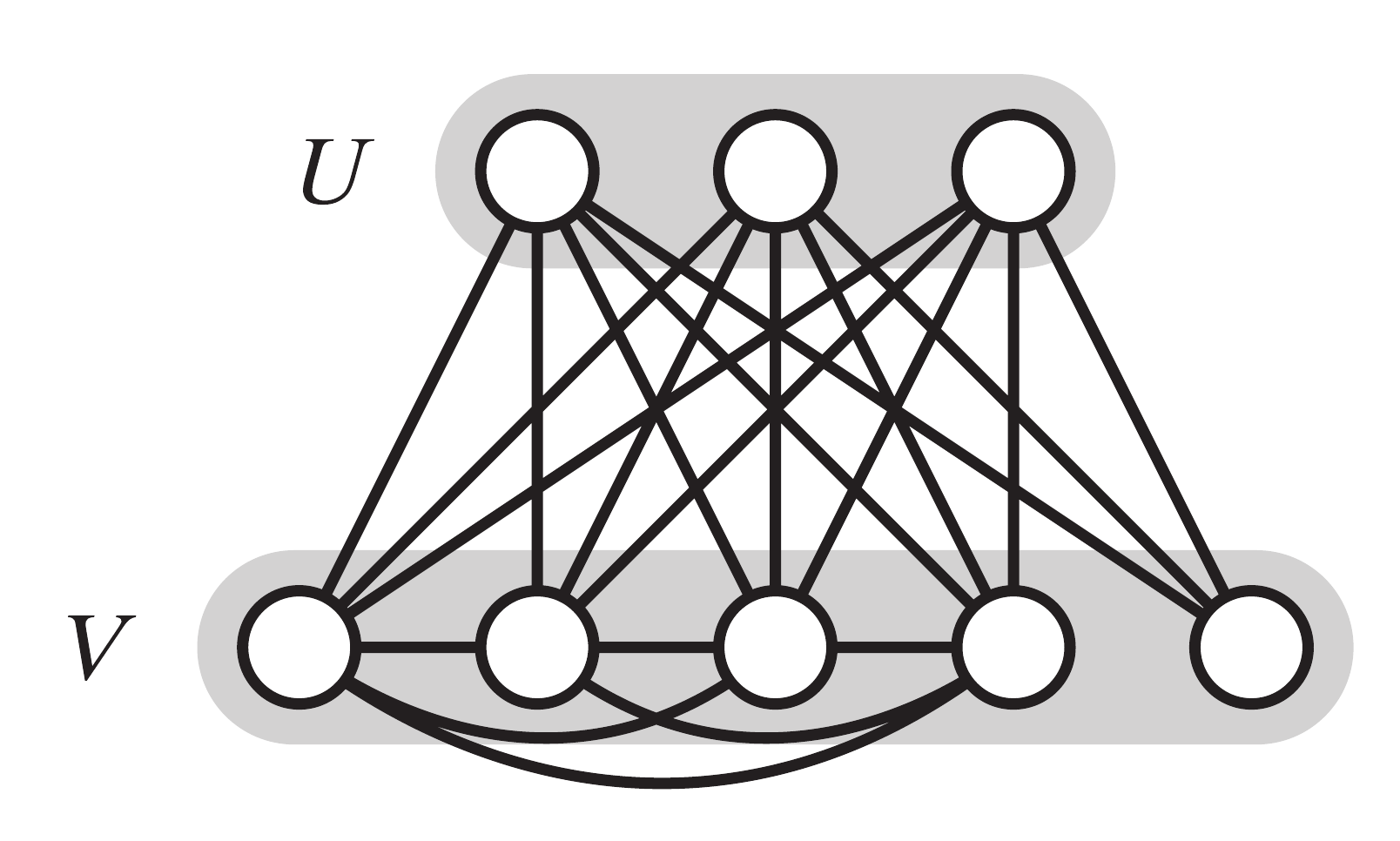}}\\
\caption{The graph $K(3,5)$.}
\label{fig:K35} 
\end{center}
\end{figure}

Note that $\kappa(K(s,t)) \leq \delta(K(s,t)) = s$. Further, note that adding edges to a graph cannot decrease its vertex connectivity, so $\kappa(K(s,t))\geq \kappa(K_{s,t})=s$. This gives $\kappa(K(s,t))=s$. By Theorem \ref{thm:thhkappa},
\begin{equation}\label{eqn:thhKst lowerbound}
\thh\left( K(s,t) \right) \geq \lceil 2\sqrt{s+t-s}+s-1 \rceil = \lceil 2\sqrt{t}+s-1 \rceil.
\end{equation}
We consider the cases where $s\geq 1$ and $s=0$ separately.

First, for any positive integer $s$, pick an integer $t$ such that $t\geq \lceil 2\sqrt{t}+s-1 \rceil$. That is, pick
\begin{alignat*}{3}
        &\qquad & t                 &\geq \lceil 2\sqrt{t}+s-1 \rceil   &\qquad \\
\iff    &       & t                 &\geq 2\sqrt{t}+s-1                 & &\text{since $t$ is an integer,}    \\
\iff    &       & t-2\sqrt{t}+1     &\geq s                             \\
\iff    &       & (\sqrt{t}-1)^2    &\geq s                             \\
\iff    &       & \sqrt{t}-1        &\geq \sqrt{s}                      & &\text{since $s$ and $t$ are positive integers,}\\
\iff    &       & \sqrt{t}          & \geq \sqrt{s}+1                   \\
\iff    &       & t                 & \geq (\sqrt{s}+1)^2               \\
\iff    &       & t                 & \geq s + 2\sqrt{s} + 1.
\end{alignat*}
We will show that no hopping forcing set of $K(s,t)$ realizes this lower bound for these $s$ and $t$.

Let $B$ be a hopping forcing set of $G=K(s,t)$. There are four cases to consider, based on whether $U$ and $V$ are subsets of $B$.

\emph{Case 1 ($U\nsubseteq B$ and $V\nsubseteq B$):} If so, then there are white vertices in both $U$ and $V$ at time $0$. Since each vertex in $U$ is adjacent to every vertex in $V$ and vice versa, every vertex in $B$ is adjacent to a white vertex and thus is dormant, contradicting $B$ being a hopping forcing set.

\emph{Case 2 ($U\subseteq B$ and $V\subseteq B$):} If so, then $|B|=n$ and $\thh(G;B)=n$. Note that for $s\geq 1$, we have $t\geq s+2\sqrt{s}+1\geq 4$. Then
\begin{alignat*}{3}
        &\qquad & t     &\geq 4         &\qquad \\
\iff    &       & t^2   &\geq 4t        &   &\text{since $t$ is positive,}    \\
\iff    &       & t     &\geq 2\sqrt{t} &   &\text{since $t^2$ and $4t$ are positive,} \\
\iff    &       & t     &> \lceil 2\sqrt{t}-1 \rceil    & &\text{since $2\sqrt{t}>\lceil 2\sqrt{t}-1 \rceil$ by definition of ceiling,}  \\
\iff    &       & s+t   &> \lceil 2\sqrt{t}-1 \rceil +s \\
\iff    &       & n     &> \lceil 2\sqrt{t}+s-1 \rceil.
\end{alignat*}
So, in this case, $\thh(G;B)=n$ is always greater than the lower bound from Equation \ref{eqn:thhKst lowerbound}.

\emph{Case 3 ($U\subseteq B$ but $V\nsubseteq B$):} Suppose as such, and that $B$ is a hopping forcing set such that there are at least two white vertices in $V$ at time $0$ (that is, $|B|\leq n-2$). First, note that every vertex in $U$ will be dormant until $V(G)$ is colored entirely blue, so only vertices in $V$ will possibly be able to force. If $v\in B$, then $v$ is active and can force one of the white vertices in $V'$ to be blue, but since at least one white vertex remains in $V'$ at time $1$, no other vertex is active. If $v\notin B$, then there is at least one white vertex in $V'$, implying that there are no active vertices in $G$. This contradicts $B$ being a hopping forcing set, so $|B|\geq n-1$. Since $V\nsubseteq B$, $|B|\neq n$ and so $|B|=n-1$.

There is thus only one white vertex at time $0$ in this case. If that white vertex is $v$, then it can be forced in the first time step by any vertex in $V'$; if that white vertex is in $V'$, then it can be forced in the first time step by $v$. Either way, we have $\thh(G;B)=n$. As in Case 2, this implies that $\thh(G;B)$ is greater than the lower bound in Equation \ref{eqn:thhKst lowerbound}.

\emph{Case 4 ($U\nsubseteq B$ but $V\subseteq B$):} First, observe that every vertex in $V$ will be dormant until $V(G)$ is colored entirely blue, so only vertices in $U$ will be able to force. As such, since $B$ must contain the $t$ vertices in $V$ but coloring only $V$ blue will not give any active vertices, $|B|\geq t+1$. In fact, the smallest possible hopping forcing set in this case has size $t+1$; by coloring $V$ plus one vertex in $U$ blue, we can force the rest of the vertices in $U$ one-by-one. Considering this, we must have that
\[
\thh(G;B)\geq t+1 > t \geq \lceil 2\sqrt{t}+s-1 \rceil.
\]
Thus, in this case, $\thh(G;B)$ is always strictly greater than the lower bound from Equation \ref{eqn:thhKst lowerbound}.

Altogether, in every case for $s\geq 1$, we have that $\thh(G;B)> \lceil 2\sqrt{t}+s-1 \rceil$, so we must also have that $\thh(G) > \lceil 2\sqrt{t}+s-1 \rceil$.

We now consider the case where $s=0$. If so, then notice that $G=K(0,t)=K_{t-1} \cup K_1$. Pick any $t\geq 4$. By the same logic from Case 3 above, the only hopping forcing sets of $G$ have at least $t-1$ vertices, so $\thh(G)=t$. Observe that
\begin{alignat*}{3}
        &\qquad & t     &\geq 4         &\qquad \\
\iff    &       & t^2   &\geq 4t        &   &\text{since $t$ is positive,}    \\
\iff    &       & t     &\geq 2\sqrt{t} &   &\text{since $t^2$ and $4t$ are positive,} \\
\iff    &       & t     &> \lceil 2\sqrt{t}-1 \rceil    & &\text{since $2\sqrt{t}>\lceil 2\sqrt{t}-1 \rceil$ by definition of ceiling,}  \\
\iff    &       & t     &> \lceil 2\sqrt{t} +s-1 \rceil & &\text{since $s=0$.}
\end{alignat*}
Thus, when $s=0$, $\thh(G)=t$ is strictly greater than the lower bound from Equation \ref{eqn:thhKst lowerbound}.

For any $\kappa\geq 0$, there are then infinitely many graphs that do not realize the bound from Theorem \ref{thm:thhkappa}, those being the graphs $K(s,t)$ where $s=\kappa$ and $t\geq \max\{4, s+2\sqrt{s}+1\}$.
\end{proof}


\subsection{Upper bounds}\label{sec:throtUpperBounds}

In this section, we provide three ways to bound the hopping throttling number from above. First, consider the hopping forcing set $B=V(G)$ for any graph $G$. Then, since $\thh(G;B)=|V(G)|$, we immediately get the following elementary bound on $\thh(G)$.

\begin{obs}\label{obs:thhleqn}
For any graph $G$, $\thh(G)\leq |V(G)|$.
\end{obs}

The second bound generalizes the strategy used to give an upper bound for $\thh(K_{s,t})$ in the proof of Proposition \ref{prop:thhkappatight}. In that proof, we colored the vertices in the smaller of the partite sets blue, then optimized throttling on the remaining vertices as we would on an empty graph. To generalize this, we will consider a maximum independent set of a graph, which, as with the partite sets of $K_{s,t}$, induces an empty graph.

\begin{prop}\label{prop:thhalpha}
If $G$ is a graph on $n$ vertices with independence number $\alpha$, then
\[
\thh(G)\leq \lceil n-\alpha + 2\sqrt{\alpha} - 1\rceil.
\]
\end{prop}
\begin{proof}
Let $S\subseteq V(G)$ be a maximum independent set of $G$; then, $|S|=\alpha$. Color all $n-\alpha$ vertices in $V(G)\setminus S$ blue. Since the remaining $\alpha$ vertices in $S$ form an independent set, no two of them neighbor each other. Thus, every vertex in $S$ neighbors only blue vertices from $V(G)\setminus S$. Coloring some vertices in $S$ blue, then, will allow us to force other vertices in $S$ to become blue by hopping.

It remains then to optimize throttling on the vertices in $S$. Since $S$ is an independent set, $G[S]$ is isomorphic to $\ol{K_\alpha}$. Thus, by Proposition \ref{prop:thhempty}, $\thh(G[S])=\lceil 2\sqrt{\alpha}-1 \rceil$.

Altogether, we have that $\thh(G)\leq n-\alpha + \lceil 2\sqrt{\alpha}-1 \rceil = \lceil n-\alpha + 2\sqrt{\alpha} - 1\rceil$ as desired.
\end{proof}

Note that this upper bound also holds for $\Zf$ throttling, as demonstrated in \cite[Proposition 5.5]{C19}, since $\Zf$ forcing allows us to use the hopping color change rule in addition to the standard one.

As with the lower bound in Theorem \ref{thm:thhkappa}, there are infinitely many graphs for which this upper bound is tight. In fact, the complete bipartite graphs $K_{s,t}$ serve as an example of tightness for both the $\kappa$ lower bound and this $\alpha$ upper bound. To show this, we first prove a lemma that demonstrates a case in which our lower and upper bounds are equal.

\begin{lem}\label{lem:k+a=n}
For a graph $G$ of order $n$, vertex connectivity $\kappa$, and independence number $\alpha$, if $\kappa+\alpha=n$, then
\[
\thh(G)= \lceil 2\sqrt{n-\kappa}+\kappa-1 \rceil = \lceil n-\alpha + 2\sqrt{\alpha} - 1\rceil.
\]
\end{lem}
\begin{proof}
We have that $\kappa=n-\alpha$. By Theorem \ref{thm:thhkappa},
\[
\thh(G) \geq \lceil 2\sqrt{n-\kappa}+\kappa-1 \rceil
        = \lceil 2\sqrt{n-(n-\alpha)}+n-\alpha-1 \rceil
        = \lceil n-\alpha + 2\sqrt{\alpha} - 1 \rceil.
\]
By Proposition \ref{prop:thhalpha}, $\thh(G)\leq \lceil n-\alpha + 2\sqrt{\alpha} - 1 \rceil$, so we have equality.
\end{proof}

Note that the converse of Lemma \ref{lem:k+a=n} does not always hold. The cross graph, illustrated below in Figure \ref{fig:crossgraph}, has $\kappa(G)=1$ and $\alpha(G)=4$, giving both bounds equal to $5$, but $\kappa(G)+\alpha(G)=5\neq 6=|V(G)|$.

\begin{figure}[ht] \begin{center}
\scalebox{.4}{\includegraphics{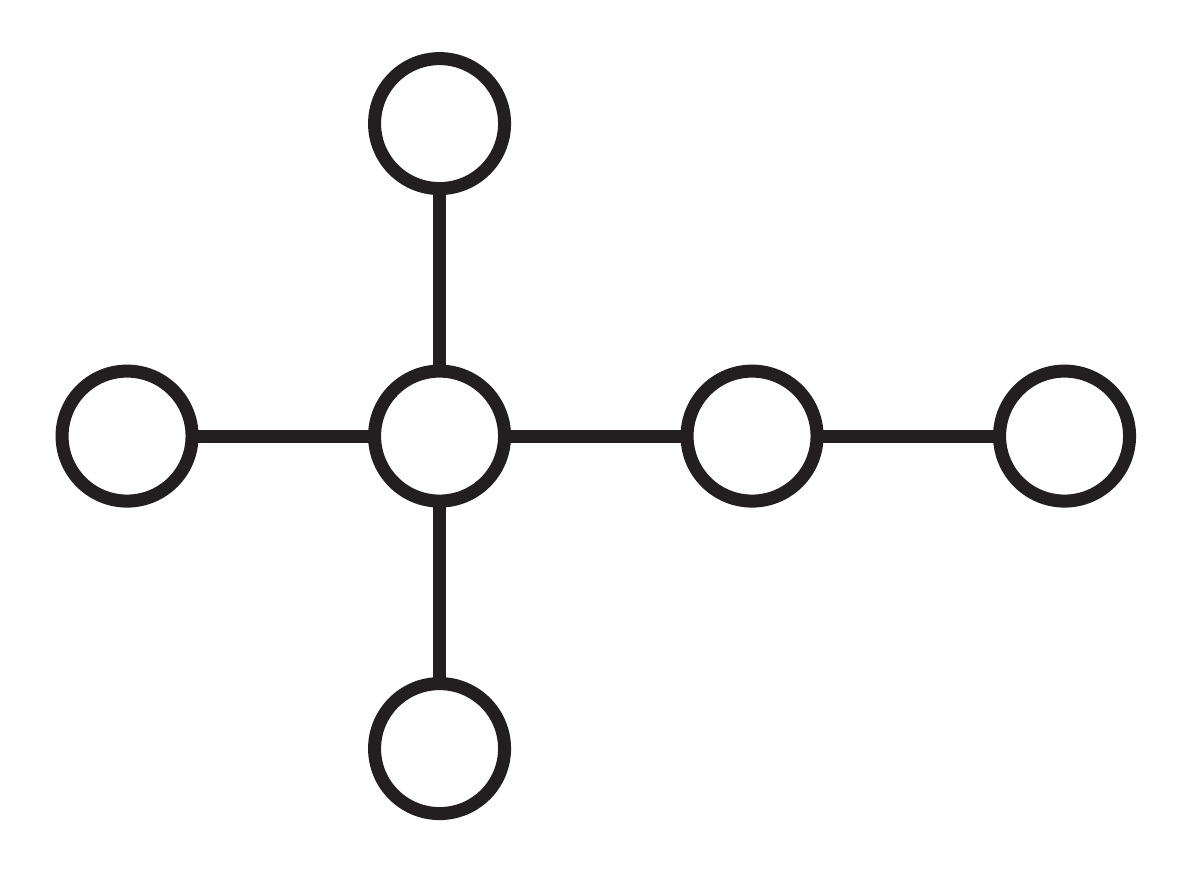}}\\
\caption{The cross graph.}
\label{fig:crossgraph} 
\end{center}
\end{figure}

Nevertheless, Lemma \ref{lem:k+a=n} lets us immediately show that $\thh(K_{s,t})$ equals the upper bound from Proposition \ref{prop:thhalpha}.

\begin{prop}\label{prop:thhalphatight}
The bound in Proposition \ref{prop:thhalpha} is tight for all $\alpha\geq 1$.
\end{prop}
\begin{proof}
Consider the complete bipartite graph $K_{s,t}$ on $s+t$ vertices with $s\leq t$, $s\geq 0$, and $t\geq 1$. Note that $\kappa(K_{s,t})= \min\{s,t\}=s$ and $\alpha(K_{s,t}) = \max\{s,t\} = t$. Since
\[
\kappa(K_{s,t}) + \alpha(K_{s,t}) = s+t = |V(G)|,
\]
by Lemma \ref{lem:k+a=n}, we have that
$
\thh(K_{s,t}) = \lceil n-\alpha + 2\sqrt{\alpha} - 1\rceil,
$
so the bound is tight.
\end{proof}




We can see Proposition \ref{prop:thhalphatight} in action in Figure \ref{fig:k35forcing}. For instance, by coloring all but a maximum independent set of vertices (namely those in $V(G)\setminus V = U$) blue, optimizing throttling on the remaining vertices (those in $V$) realizes the hopping throttling number of $\thh(K_{3,5})=7$.







We now turn our attention to our third way to bound $\thh$ from above---by using a characterization of hopping throttling. Standard throttling was first characterized in \cite{C19}, involving particular graph operations performed on the Cartesian product $K_a \cart P_b$. Define the \emph{path edges} of $K_a \cart P_b$ to be the edges in each copy of $P_b$ in the Cartesian product, and similarly define the \emph{complete edges} of $K_a \cart P_b$ to be the edges in each copy of $K_a$ in the Cartesian product. We then have the following theorem characterizing $\thz$.

\begin{thm}[{\cite[Theorem 4.1]{C19}}]\label{thm:standardchar}
Given a graph $G$ and a positive integer $t$, $\thz(G) \leq t$ if and only if there exist integers $a\geq 1$ and $b\geq 0$ such that $a+b=t$ and $G$ can be obtained from $K_a \cart P_{b+1}$ by contracting path edges and deleting complete edges.
\end{thm}

In standard zero forcing, forcing chains are always paths. This provides the intuition behind the choice of $K_a \cart P_{b+1}$ as the graph of concern in the above theorem, as each forcing chain will induce a path. In hopping forcing, however, every hopping forcing chain is an independent set of vertices, since a vertex cannot force its neighbors by hopping. As such, the Cartesian product we want to consider for the characterization theorem does not involve paths on $b+1$ vertices, but instead empty graphs on $b+1$ vertices, which are isomorphic to independent sets. Specifically, we will consider $K_a \cart \ol{K_{b+1}}$.

We also need to specify the corresponding graph operations in the characterization theorem for hopping throttling. Suppose $G=K_a \cart \ol{K_{b+1}}$ is drawn so that $V(G)$ is arranged in an $a$ by $b+1$ array where each column induces a $K_a$ and each row induces a $\ol{K_{b+1}}$. Then, as in the characterization theorem for standard throttling, deleting complete edges (the vertical edges in each copy of $K_a$) remains a valid operation. However, we can no longer contract path edges, as there are no such ``path edges'' in $G$. Instead, define an \emph{empty pair} to be two vertices in the same row of $G$ and in adjacent columns of $G$, that is, two vertices in the same copy of $\ol{K_{b+1}}$ and in adjacent copies of $K_a$. Then, contracting a path edge in $K_a \cart P_{b+1}$ corresponds exactly to identifying the vertices in an empty pair in $G$; we refer to this operation as \emph{identifying an empty pair}. This gives us the analogous characterization theorem for hopping throttling.

\begin{thm}\label{thm:hoppingchar}
Given a graph $G$ and a positive integer $t$, $\thh(G)\leq t$ if and only if there exist integers $a\geq 1$ and $b\geq 0$ such that $a+b=t$ and $G$ can be obtained from $K_a \cart \ol{K_{b+1}}$ by identifying empty pairs and deleting complete edges. 
\end{thm}
\begin{proof}
Observe that by considering hopping throttling instead of standard throttling and the graph $K_a \cart \ol{K_{b+1}}$ instead of $K_a \cart P_{b+1}$, the graph operations in this theorem are exactly analogous to the operations in the standard throttling case in Theorem \ref{thm:standardchar}. As such, the proof for this theorem is identical to that provided for Theorem \ref{thm:standardchar} in \cite{C19}, except standard zero forcing parameters are replaced with their hopping forcing equivalents, $K_a \cart P_{b+1}$ is replaced with $K_a \cart \ol{K_{b+1}}$, and the operation of contracting path edges is replaced with identifying empty pairs.
\end{proof}

It can sometimes be difficult to directly apply Theorem \ref{thm:hoppingchar} to a specific graph without calculating its throttling number beforehand in the first place. However, the intuition behind the characterization theorem is very useful. Carlson briefly mentions in \cite{C19} how the snaking strategy used to find the standard and $\Zf$ throttling numbers of paths and cycles is secretly their respective characterization theorems in disguise, and it is no different for snaking under hopping throttling. For example, Figure \ref{fig:p15charthm} illustrates how we can transform $K_4 \cart \ol{K_{5}}$ into the snaked $P_{15}$ we saw in Figure \ref{fig:p15snaking}. By deleting the grey complete edges and identifying the pairs of vertices labeled with the same number, we can then squish the resulting path into the $3\times 5$ box we saw earlier.

\begin{figure}[ht] \begin{center}
\scalebox{.26}{\includegraphics{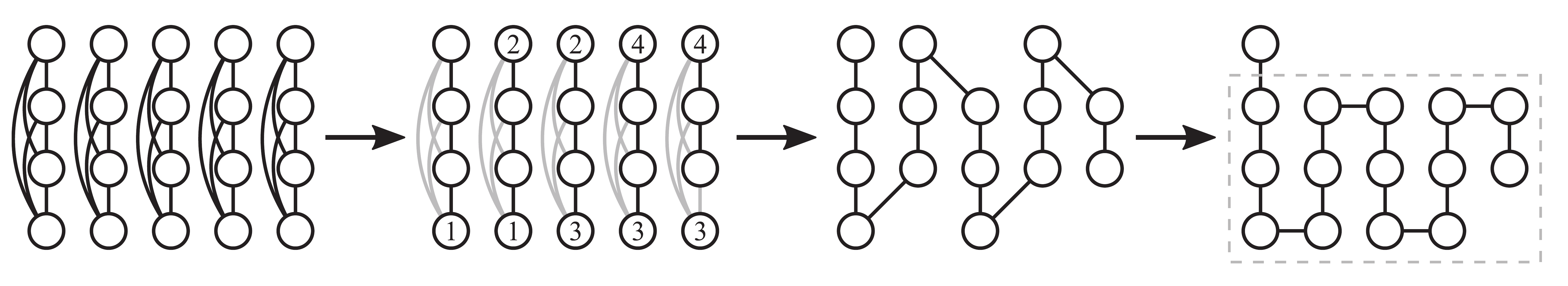}}\\
\caption{Obtaining $P_{15}$ from $K_4 \cart \ol{K_{5}}$ by identifying empty pairs and deleting complete edges.}
\label{fig:p15charthm} 
\end{center}
\end{figure}

\subsection{Extreme hopping throttling}\label{sec:throtExtreme}

We now investigate graphs with extreme hopping throttling numbers. First, we determine all graphs with $\thh(G)\in \{ 1,2,3,4 \}$ by making use, in part, of our hopping throttling characterization theorem. Afterwards, we show how considering sets of forbidden subgraphs lets us classify graphs with $\thh(G) = |V(G)| - k$.

\subsubsection{Low hopping throttling number}

We begin with a simple observation. The only way that $\thh(G)=1$ is if there is a hopping forcing set $B$ of size $1$ that forces the entire graph to be blue in $0$ time steps; this necessitates $G$ being a single vertex.

\begin{obs}
A graph $G$ satisfies $\thh(G)=1$ if and only if $G\iso K_1$.
\end{obs}

A similar consideration of possible sizes of hopping forcing sets and their resulting propagation times lets us determine all graphs with hopping throttling number $2$.

\begin{prop}
For any graph $G$ with $|V(G)|\geq 2$, $\thh(G)=2$ if and only if $G\iso K_2$ or $G\iso \ol{K_2}$.
\end{prop}
\begin{proof}
First, observe that if $G\iso K_2$, then by Proposition \ref{prop:thhpath}, $\thh(G) = \lceill 2\sqrt{2-1} \rceill = 2$. If $G\iso\ol{K_2}$, then the only hopping forcing sets are either one or both vertices in $G$; in either case, $\thh(G;B)=2$.

Now, suppose that $\thh(G;B)=2$ for some hopping forcing set $B$. Then, $|B|\in \{1,2\}$. If $|B|=1$, the propagation time is $1$. So, $v\notin B$ implies $v\notin N(B)$; further, at most one vertex can be forced in the first time step. Together, this gives $G\iso\ol{K_2}$. If $|B|=2$, the propagation time is $0$, so we must have $|V(G)|=2$ and either $G\iso K_2$ or $G\iso\ol{K_2}$.
\end{proof}

Continuing on, we can use Theorem \ref{thm:hoppingchar}, the characterization theorem, to help find every graph with a hopping throttling number of either $3$ or $4$ via a brute force algorithm. Let $t\in\{3,4\}$. For each pair of integers $a\geq 1$ and $b\geq 0$ with $a+b=t$, we start by creating the graph $G' = K_a \cart \ol{K_{b+1}}$. Next, we make a list of every possible combination of a subset of empty pairs of $G'$ and a subset of complete edges of $G'$. For each combination, we delete the subset of complete edges and identify each empty pair in the subset, giving us a graph $G$ with $\thh(G)\leq t$. If $G$ is not isomorphic to a graph we previously generated, we then add it to a running list of graphs with $\thh(G)\leq t$. Finally, we can remove any graph which also has $\thh(G)\leq t-1$ to get a list of graphs with hopping throttling number $t$.

Implementing this using Mathematica \cite{github}, we get $7$ graphs with $\thh(G)=3$ and $35$ graphs with $\thh(G)=4$. 

\subsubsection{High hopping throttling number}

We now turn to graphs with hopping throttling numbers close to $|V(G)|$. Specifically, we will consider how we can classify graphs with hopping throttling numbers $\thh(G)=|V(G)|-k$ for any integer $0\leq k\leq |V(G)|$. Similar classifications have already been obtained for other types of throttling, namely positive semidefinite throttling and standard throttling (see \cite{CHK19} and \cite{CK20} respectively). Even more noteworthy, this was done by characterizing throttling numbers in terms of forbidden subgraphs. In this section, we proceed similarly, providing explicit sets of forbidden subgraphs to classify graphs with $\thh(G)=|V(G)|$ and graphs with $\thh(G)=|V(G)|-1$, as well as providing a method to theoretically determine forbidden subgraphs for every graph with $\thh(G)=|V(G)|-k$.

To begin, we show that three forbidden subgraphs can be used to characterize graphs with $\thh(G)=|V(G)|$.

\begin{thm}
For any graph $G$, $\thh(G)=|V(G)|$ if and only if $G$ does not contain $2K_2$, $K_2\cup \ol{K_2}$, or $\ol{K_4}$ as an induced subgraph.
\end{thm}

\begin{figure}[ht] \begin{center}
\scalebox{.4}{\includegraphics{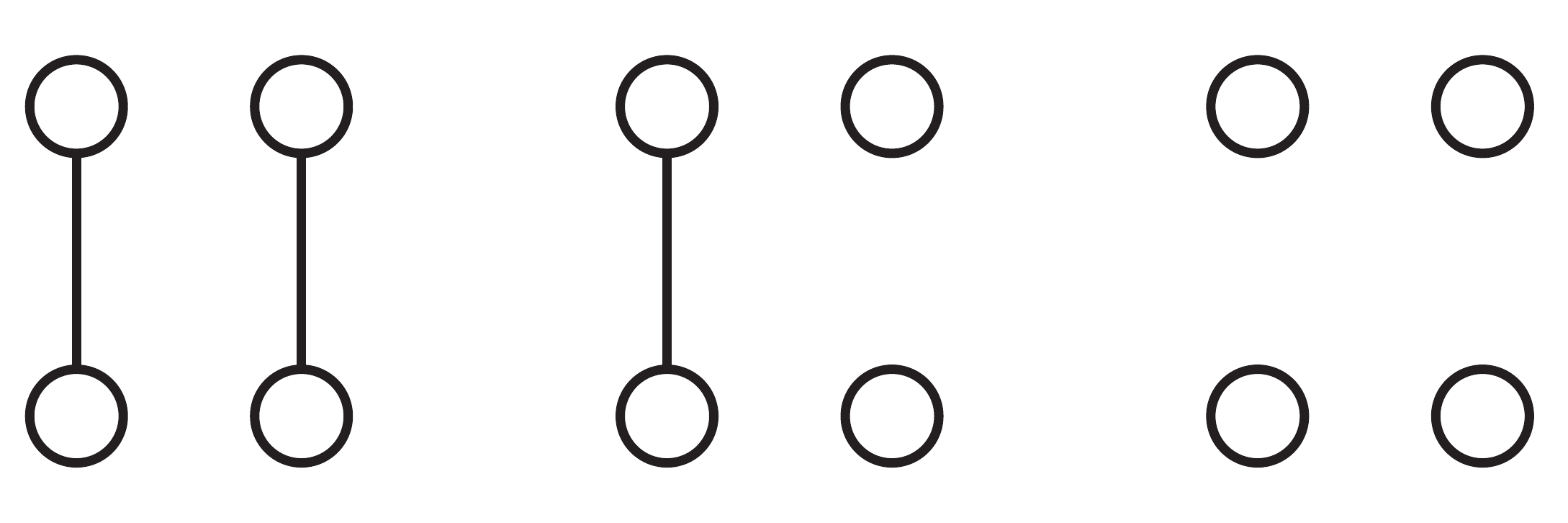}}\\
\caption{The graphs $2K_2$, $K_2\cup \ol{K_2}$, and $\ol{K_4}$ from left to right.}
\label{fig:g0subgraphs} 
\end{center}
\end{figure}

\begin{proof}
We prove both directions by contrapositive.

By Observation \ref{obs:thhleqn}, $\thh(G)\leq |V(G)|$ for any graph $G$. Let $B$ be a hopping forcing set such that $\thh(G;B)=\thh(G)$, and let $\F$ be a set of hopping forces of $B$ such that $\pth(G;\F)=\pth(G;B)$. Suppose that $\thh(G)< |V(G)|$. Then, there must exist a time $t$ such that $|\F^{(t)}| \geq 2$. Let $u,v\in \F^{(t)}$ and $x,y\in U_\F^{(t)}$ such that $x\rightarrow u$ and $y\rightarrow v$ in time step $t$. By the hopping color change rule, this implies that $xu,yv,xv,yu \notin E(G)$.

Consider the edges $xy$ and $uv$. First, assume $xy\in E(G)$. If $uv\in E(G)$, then $S=\{u,v,x,y\}$ induces a $2K_2$; if $uv\notin E(G)$, then $S$ induces a $K_2\cup \ol{K_2}$. Now, assume $xy\notin E(G)$. If $uv\in E(G)$, then $S$ induces a $K_2\cup \ol{K_2}$; if $uv\notin E(G)$, then $S$ induces a $\ol{K_4}$. In all cases, if $\thh(G)< |V(G)|$, then $G$ contains an induced $2K_2$, $K_2\cup \ol{K_2}$, or $\ol{K_4}$.

For the converse, suppose that $G$ contains $2K_2$, $K_2\cup \ol{K_2}$, or $\ol{K_4}$ as an induced subgraph. In any case, there exist four vertices $u,v,x,y\in V(G)$ such that $xu,yv,xv,yu\notin E(G)$. Let $B=V(G) \setminus \{ u,v \}$; $x$ and $y$ are then active, since both are blue and not adjacent to the only vertices colored white in $G$, namely $u$ and $v$. As such, in the first time step of the hopping forcing process, $x$ and $y$ can force $u$ and $v$ respectively to color the entire graph blue. This gives $\thh(G)\leq |B|+\pth(G;B)=|V(G)|-1$.
\end{proof}

We now know that by considering a set of three forbidden subgraphs, we can determine whether or not a graph $G$ has hopping throttling number $|V(G)|$. We will now generalize this to describe sets of forbidden subgraphs for $\thh(G)\geq |V(G)|-k$. As mentioned before, Carlson and Kritschgau did just this for standard throttling in \cite{CK20}, and the method described here to find these forbidden subgraphs for hopping throttling is very similar. To begin, consider the following definition; recall that a \emph{matching} is a set of edges such that no two share a common vertex.

\begin{defn}[{\cite[Definition 4.8]{CK20}}]
A graph $G$ is an \emph{$a$-accelerator} for integer $a\geq 1$ if $V(G)$ can be partitioned into sets $S$ and $T$, each of size $a+1$, such that there exists a matching between $S$ and $T$, and the only edges between $S$ and $T$ are in this matching.
\end{defn}

Observe then that if $G$ is an $a$-accelerator, by coloring $S$ blue, we can force $T$ to be blue in one time step. This gives $\thh(G)\leq |S|+1=k+2=|V(G)|-k$. To ensure the same result under the hopping color change rule, we must ensure that no edges connect a blue vertex in $S$ to a white vertex in $T$. This motivates the following definition.

\begin{defn}
A graph $G$ is a \emph{$k$-kangaroo} for integer $k\geq 1$ if $V(G)$ can be partitioned into sets $S$ and $T$, each of size $k+1$, such that no edges go between $S$ and $T$.
\end{defn}

Thus, just as with $a$-accelerators, if $G$ is a $k$-kangaroo, then by coloring $S$ blue, $T$ can be forced in one time step, giving $\thh(G)\leq |V(G)|-k$. We can generalize this idea to hopping forcing processes with multiple time steps by gluing $k$-kangaroos together.

\begin{defn}
A graph $G$ is a \emph{$(k_1,\ldots,k_r)$-kangaroo} for positive integers $k_1,\ldots,k_r$ if $V(G)$ can be written as $\bigcup_{i=1}^r (S_i \cup T_i)$, where $S_i$ and $T_i$ are sets of $k_i$ vertices for $1\leq i \leq r$, such that
\begin{enumerate}
    \item $\{S_i \}_{i=1}^r$ is a set of disjoint sets,
    \item $\{T_i \}_{i=1}^r$ is a set of disjoint sets, and
    \item $S_i \cap T_j$ is empty whenever $i \leq j$.
\end{enumerate}
Furthermore, the edges of $G$ must be partitioned by $S_i,T_i$ for $1\leq i\leq r$ such that
\begin{enumerate}[resume]
    \item no edges go between $S_i$ and $T_j$ for $i\leq j$, and
    \item $S_i$ is dominated by $T_{i-1}$ for $2\leq i\leq r$.
\end{enumerate}
\end{defn}

Note that property 1 ensures no vertex forces twice, property 2 ensures no vertex is forced twice, and property 3 ensures no vertex can force before it itself is forced. Property 4 ensures that the vertices in each $S_i$ are not adjacent to any white vertices at time $i-1$, giving us that $G[S_i \cup T_i]$ is a $k_i$-kangaroo for every $i$. Further, by property 5, $B^{(i)}=T_i$. Figure \ref{fig:kangarooex} gives an example of a $(2,3,1)$-kangaroo.

\begin{figure}[ht] \begin{center}
\scalebox{.35}{\includegraphics{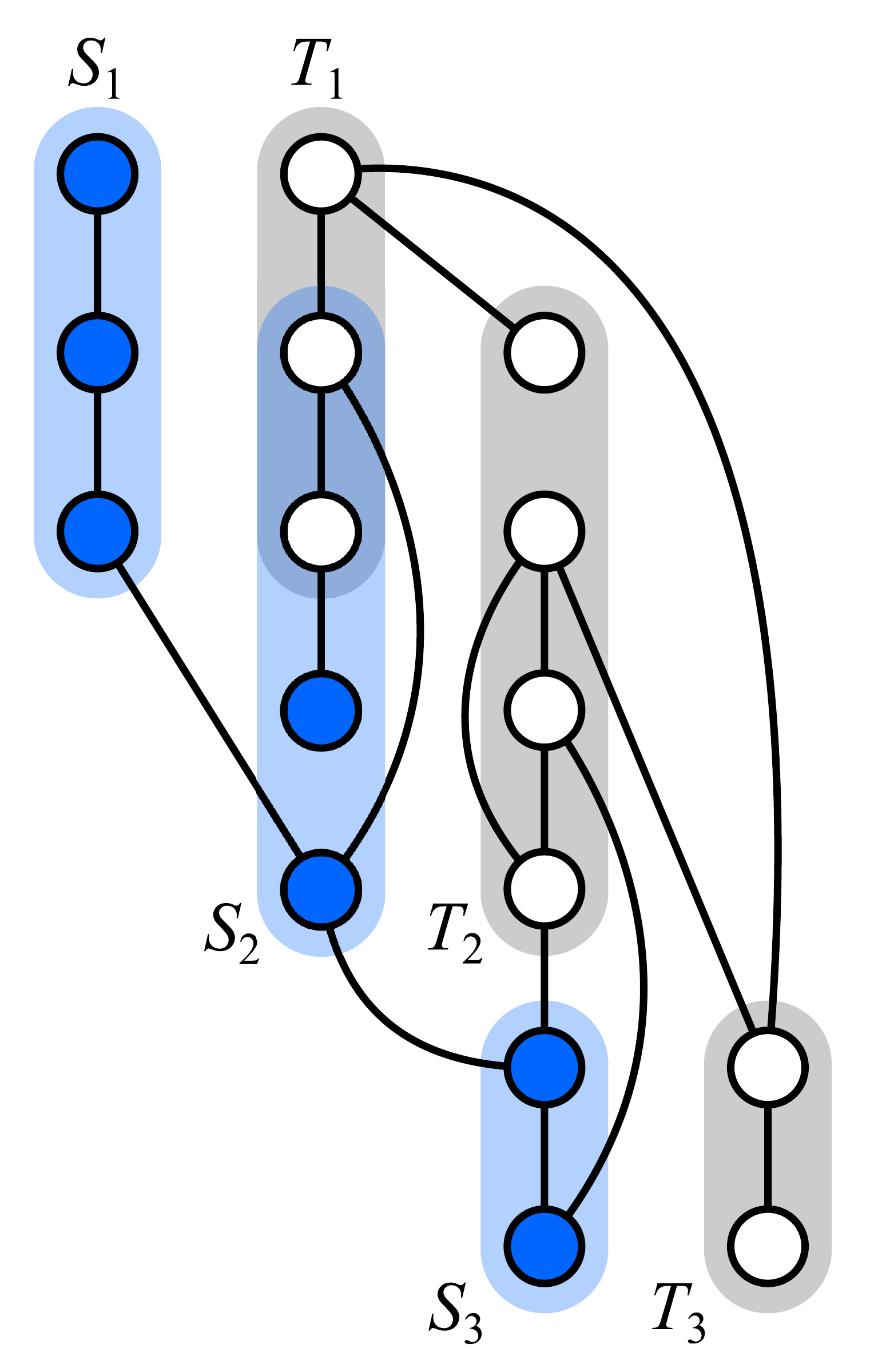}}\\
\caption{A $(2,3,1)$-kangaroo, with a hopping forcing set in blue and sets $S_i$ and $T_i$ highlighted and labeled. Note that $T_1\cap S_2$ is nonempty.}
\label{fig:kangarooex} 
\end{center}
\end{figure}

Define $\mc{M}_{k_1,\ldots,k_r}$ to be the set of $(k_1,\ldots,k_r)$-kangaroos. Then, just as the hopping throttling number of a $k$-kangaroo is bounded above by $|V(G)|-k$, we can bound the hopping throttling number of a $(k_1, \ldots, k_r)$-kangaroo from above by considering the sum of the $k_i$ values.

\begin{obs}\label{obs:thhkangaroo}
If $M\in \mc{M}_{k_1,\ldots,k_r}$, then
\[
\thh(M) < |V(M)| - (k_1 + \ldots + k_r) + 1.
\]
\end{obs}

Observe that if a graph $G$ contains a graph in $\mc{M}_{k_1,\ldots,k_r}$ as an induced subgraph, then the throttling number of $G$ must also be bounded by the upper bound in Observation \ref{obs:thhkangaroo}.
Define \[
\mc{G}_k = \bigcup_{k_1+\ldots+k_r=k+1} \mc{M}_{k_1,\ldots,k_r}.
\]
We then have the following theorem.

\begin{thm}
Let $G$ be a graph. Then, $\thh(G)\geq |V(G)|-k$ if and only if $G$ does not contain a graph in $\mc{G}_k$ as an induced subgraph.
\end{thm}

The proof for this theorem is analogous to that of \cite[Theorem 4.11]{CK20}, substituting standard throttling for hopping throttling and $(a_1, \ldots, a_r)$-accelerators for $(k_1,\ldots,k_r)$-kangaroos, so we omit it here. That said, this theorem provides us with a crucial corollary, one that allows us to classify exactly when $\thh(G)=|V(G)|-k$.

\begin{cor}
Let $G$ be a graph. Then, $\thh(G) = |V(G)|-k$ if and only if $G$ does not contain a graph in $\mc{G}_k$ as an induced subgraph, but contains a graph in $\mc{G}_{k-1}$ as an induced subgraph.
\end{cor}

Using the definitions and results above, it is straightforward to algorithmically find all graphs in $\mc{G}_k$ by first generating all graphs in $\mc{M}_{k_1,\ldots,k_r}$ where $\sum k_i = k+1$, then removing any possible isomorphisms. Furthermore, we can then remove any graphs from $\mc{G}_k$ that are supergraphs of another graph in $\mc{G}_k$ to eliminate more redundancy. Implementing this process using the Python code available in \cite{github} for $\mc{G}_1$ gives us $108$ forbidden subgraphs for $\thh(G)=|V(G)|-1$. 

\subsection{Comparing hopping throttling to other types of throttling}\label{sec:comparingThrottling}

In this section, we consider how the hopping throttling number relates to other throttling numbers, namely those for standard zero forcing and $\Zf$ forcing. First, we will show that $\thh(G)$ and $\thz(G)$ are incomparable; that is, the gap between $\thh(G)$ and $\thz(G)$ can be made arbitrarily large in either direction. To do so, we consider a few examples of graphs with significant differences between their hopping throttling and standard throttling numbers. The first graph we examine is the Petersen graph, which exhibits a gap of two between its standard and hopping throttling numbers.

\begin{ex}\label{ex:throttlingpetersen}
Consider the Petersen graph $P$. By \cite[Proposition 3.26]{AIM08}, $\Z(P)=5$, so $\thz(P)$ is at least $6$. By coloring the five vertices in the outer cycle of $P$ blue, we can force the inner cycle to become blue in the first time step, giving us $\thz(P)=6$.

Note that $\kappa(P)=3$, so by Theorem \ref{thm:thhkappa}, $\thh(P) \geq \lceil 2\sqrt{10-3}+3-1 \rceil = 8$. Together with the forcing process illustrated in Figure \ref{fig:petersenforcing}, we have that $\thh(P)=8$.
\end{ex}

\begin{figure}[ht] \begin{center}
\scalebox{.4}{\includegraphics{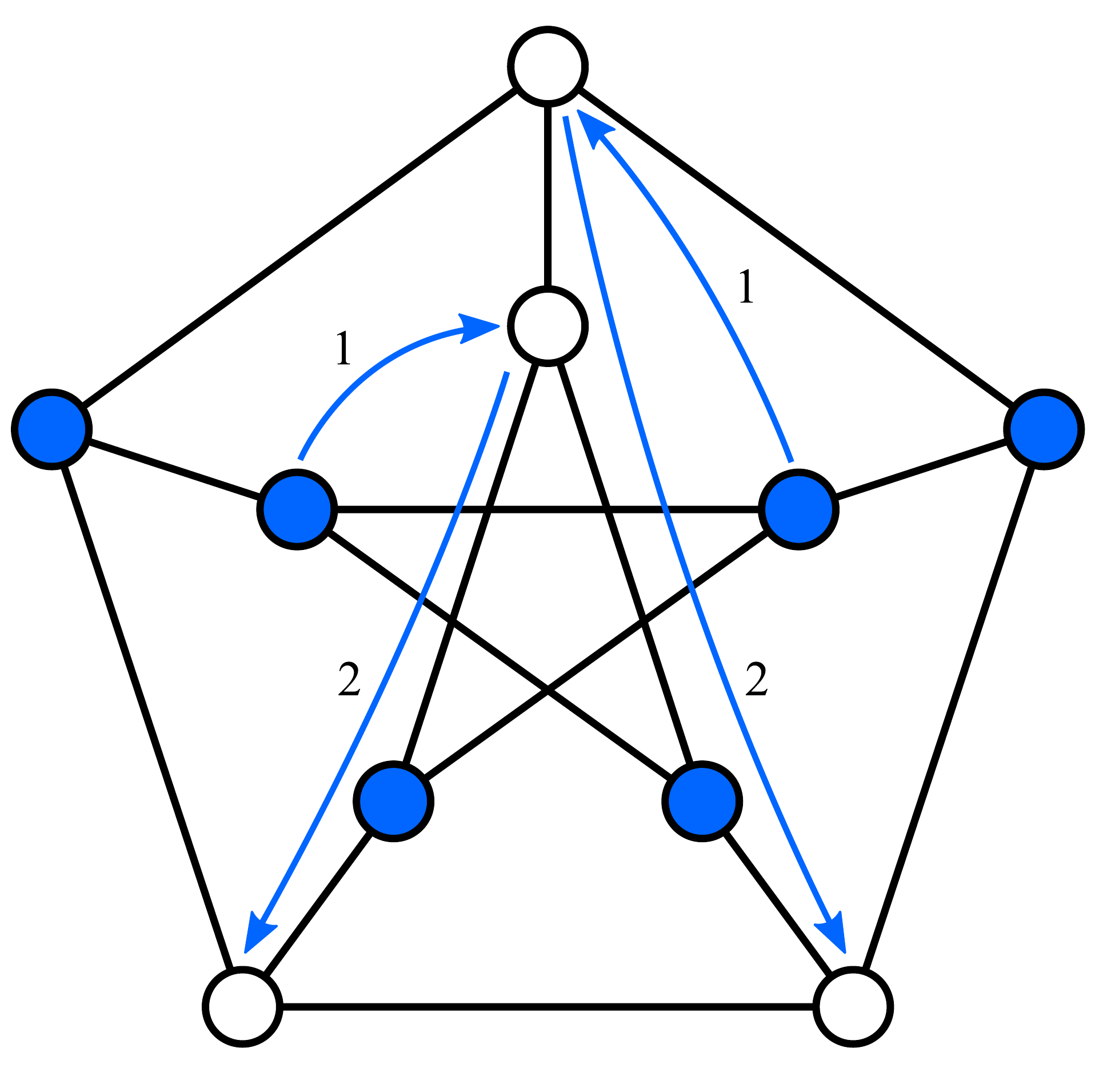}}\\
\caption{The hopping forcing process on the Petersen graph.}
\label{fig:petersenforcing} 
\end{center}
\end{figure}

As we have just seen, the Petersen graph has $\thh(G)>\thz(G)$; we now show that there are infinitely many graphs with this property by determining the throttling numbers of the complete bipartite graphs.

\begin{ex}\label{ex:throttlingKst}
Consider once again the complete bipartite graph $K_{s,t}$ on $s+t$ vertices with partite sets $U$ and $V$ such that $2\leq s=|U|\leq |V|=t$. By Proposition \ref{prop:forcingKst}, $\Z(K_{s,t})=s+t-2$. Furthermore, by coloring every vertex but one in each partite set blue, we can force the remaining two vertices to become blue in the first time step, implying that $\thz(K_{s,t})=s+t-1$.

In addition, by Proposition \ref{prop:thhkappatight}, $\thh(K_{s,t}) = \lceil 2\sqrt{t}+s-1 \rceil$.
\end{ex}



In fact, observe that by these throttling numbers, the gap between $\thh(K_{s,t})$ and $\thz(K_{s,t})$ can be made arbitrary large. Now, to show that there are infinitely many graphs with $\thz(G)>\thh(G)$, we use the following family.

\begin{ex}\label{ex:throttlingksp2}
Consider the family of graphs $K_s \cart P_2$ on $2s$ vertices. By \cite[Proposition 3.3]{AIM08}, $\Z(K_s \cart P_2)=s$, so $\thz(K_s \cart P_2)\geq s+1$. By coloring all $s$ vertices in one copy of $K_s$ blue, we can force the rest of the vertices to become blue in the first time step, implying that $\thz(K_s \cart P_2) = s+1$.

Observe that $\kappa(K_s \cart P_2)=s$, and so by Theorem \ref{thm:thhkappa}, \[
\thh(K_s \cart P_2) \geq \lceil 2\sqrt{2s-s}+s-1 \rceil = \lceil 2\sqrt{s}+s-1 \rceil.
\]
To show this is also an upper bound, consider a hopping forcing process in which we color one copy of $K_s$ entirely blue, then throttle on the remaining vertices by coloring some blue, and letting the vertices they are matched to do the forcing. Using Proposition \ref{prop:thhempty}, this gives $\thh(K_s \cart P_2)\leq s+\lceil 2\sqrt{s} - 1 \rceil = \lceil 2\sqrt{s}+s-1 \rceil$.
\end{ex}

The previous example, as with the complete bipartite graphs, shows that the gap between $\thh(K_s \cart P_2)$ and $\thz(K_s \cart P_2)$ can be made arbitrarily far apart, albeit in the other direction. So, combining Examples \ref{ex:throttlingKst} and \ref{ex:throttlingksp2}, we arrive at the following remark.

\begin{rmk}
The gap between $\thh(G)$ and $\thz(G)$ can be made arbitrarily far apart in both directions.
\end{rmk}

While the standard throttling number cannot be bounded nicely by the hopping throttling number, the $\Zf$ throttling number can. Since $\zf$ forcing involves both the standard and hopping color change rules, any zero forcing set of a graph $G$ that realizes $\thz(G)$ or any hopping forcing set that realizes $\thh(G)$ will be a $\zf$ forcing set that bounds $\thzf(G)$.

\begin{obs}
For any graph $G$, $\thzf(G)\leq \min\{ \thz(G), \thh(G) \}$.
\end{obs}


\section{Product throttling for hopping forcing}\label{sec:productThrottling}

So far, we have defined hopping throttling as minimizing the sum of the size of a hopping forcing set and its propagation time. However, the choice to add these two values is just that---a choice. Instead, we can try to minimize the size of the hopping forcing set \emph{times} its propagation time; this is known as product throttling. First introduced in \cite{BBB19} in the context of the graph theoretic game Cops and Robbers, product throttling was later applied to various types of forcing by Anderson et al.\ in \cite{ACF21}. In this section, we consider product throttling for hopping forcing, and show that it behaves very similarly to product throttling for standard zero forcing. 

First, note that for a graph $G$, non-negative integer $k$, and an abstract color change rule $\X$, $\ptx(G,k)$ is the minimum size of a subset $B \subseteq V(G)$ with $\ptx(G,B) = k$. There are two ways to define product throttling---either with or without an initial cost. We give each definition below for hopping forcing; the definitions for standard zero forcing are analogous and use a subscript $\Z$.

\begin{defn}(\cite{ACF21})
The \emph{product hopping throttling number with initial cost} of a graph $G$ is defined as
\[
\thht(G) = \min_{\H(G)\leq k\leq |V(G)|} \left\{ k \left(1+\pth(G,k) \right) \right\}.
\]
Similarly, the \emph{product hopping throttling number without initial cost} of $G$ is
\[
\thhs(G) = \min_{\H(G)\leq k < |V(G)|} \left\{ k \cdot \pth(G,k) \right\}.
\]
\end{defn}

Besides the presence of the initial cost of course, there is one notable key difference between these definitions---for no initial cost product throttling, we do not consider the case where $k=|V(G)|$. If we did, then since $\pth\left( G,|V(G)| \right)=0$, we would always have that $\thhs(G)=0$. This is not an issue for initial cost product throttling, since $1+\pth\left( G,|V(G)| \right)$ can never equal $0$ as propagation time is always positive.

In \cite[Remark 5.1]{ACF21}, Anderson et al.\ show that initial cost standard product throttling of any graph $G$ is always equal to $n=|V(G)|$, and the same argument holds for initial cost hopping product throttling. For any hopping forcing set $B$, we know by logic in the proof of Theorem \ref{thm:thhkappa} that at most $|B|$ vertices can be forced in any particular time step. After $\pth(G;B)$ time steps, all $n$ vertices of $G$ must be blue, so we have that
\[
|B|+\underbrace{|B|+\ldots+|B|}_{\pth(G;B) \text{ times}} \geq n
\iff |B| + |B|\pth(G;B) \geq n
\iff |B|\left(1+\pth(G;B) \right) \geq n.
\]
Thus, for every integer $k$ such that $\H(G)\leq k\leq n$, we have that $k \left(1+\pth(G,k) \right) \geq n$. Since $k \left(1+\pth(G,k) \right)=n$ when $k=n$, we have the following observation.

\begin{obs}
For any graph $G$, $\thht(G)=|V(G)|$.
\end{obs}

No initial cost product throttling for standard zero forcing is comparably more interesting, but can still be characterized quite succinctly. In fact, Anderson et al.\ found that $\thzs(G)$ is always equal to the smallest $k$ such that there exists a standard forcing set $B$ of size $k$ with $\ptz(G;B)=1$. Define $\kz(G,p)= \min\{|B| \mid \ptz(G;B)=p \}$.

\begin{thm}[{\cite[Theorem 5.3]{ACF21}}]\label{thm:thz*}
For any graph $G$, $\thzs(G)$ is the least $k$ such that $\ptz(G,k)=1$; i.e., $\thzs(G)=\kz(G,1)$. Necessarily $\kz(G,1)\geq \frac{n}{2}$.
\end{thm}

We will show that the same result holds for hopping. To do so, we first need to consider the notion of reversing a set of hopping forces.

\begin{defn}
Let $B$ be a hopping forcing set of a graph $G$, and let $\mc{F}$ be a set of hopping forces of $B$. Then the \emph{reverse set of hopping forces} of $\mc{F}$, denoted $\Rev(\F)$, is the set of hopping forces obtained by reversing each force in $\F$. The \emph{terminus} of $\F$, denoted $\Term(\F)$, is the set of vertices of $G$ that do not perform a force in $\F$.
\end{defn}

These sets were first defined in terms of standard forcing \cite[Definition 2.3]{HHK12}, but are similarly crucial for hopping forcing processes, as a single hopping forcing set can have many corresponding termini depending on which vertices get forced when. Despite this, no matter which set of hopping forces $\F$ of a hopping forcing set we consider, the terminus of $\F$ will itself always be a hopping forcing set.

\begin{prop}\label{prop:HFSreversing}
If $B$ is a hopping forcing set of $G$ and $\mc{F}$ is a set of hopping forces of $B$, then $\Term(\mc{F})$ is a hopping forcing set of $G$ and $\Rev(\mc{F})$ is a valid set of hopping forces of $\Term(\mc{F})$.
\end{prop}
\begin{proof}
Let
\[
\mc{L}=( u_1\rightarrow v_1,\ u_2\rightarrow v_2,\ \ldots,\ u_n \rightarrow v_n )
\]
be the chronological list of hopping forces of $B$ that yields the set of forces $\mc{F}$. We can write $\mc{L}$ in reverse order and reverse each force, giving
\[
\mc{L}'=( v_n\rightarrow u_n,\ v_{n-1}\rightarrow u_{n-1},\ \ldots,\ v_1 \rightarrow u_1 ).
\]
Let $B'=\Term(\mc{F})$; we will show that $\mc{L}'$ gives a valid chronological list of forces of $B'$.

Consider the first ``force'' in $\mc{L}'$, $v_n\rightarrow u_n$. In $\mc{L}$, the last force performed is $u_n \rightarrow v_n$. Immediately before this force was performed, $v_n$ was colored white, implying that each of the neighbors of $v_n$ were dormant at this time. As such, since $u_n \rightarrow v_n$ is the last force performed in $\mc{L}$, none of the neighbors of $v_n$ perform any forces in $\mc{L}$. Then, by definition of terminus, every neighbor of $v_n$ must be in $B'$. Thus, since $v_n$ does not perform a force in $\mc{L}$, it is also in $B'$, implying $v_n$ is active at time $0$ in $\L'$. As a result, $v_n\rightarrow u_n$ is a valid force in $\mc{L}'$. 

To show that $v_{n-1}\force u_{n-1}$ is a valid hopping force in $\L'$, consider the graph $G-v_n$. Observe that $\L \setminus \{ u_n \force v_n \}$ is a valid chronological list of hopping forces of $B$ on $G-v_n$, since we perform the same forces as on $G$ except the last, in which we would force $v_n$. So, by similar logic to the above, $v_{n-1}\force u_{n-1}$ is a valid force in $\L'$. A similar argument holds to show that every other force in $\L'$ is a valid hopping force. Thus, applying the forces in the chronological list of hopping forces $\L'$ of $B'=\Term(\F)$ forces the entire graph to become blue, and so $\Term(\F)$ is a hopping forcing set.

Since $\mc{L'}$ yields the set of hopping forces $\Rev(\mc{F})$ of $\Term(\mc{F})$, it follows that $\Rev(\mc{F})$ is a valid set of hopping forces of $\Term(\mc{F})$.
\end{proof}

As noted above, the terminus and reverse set of forces of a set of forces is also defined for standard zero forcing \cite{HHK12}. However, this idea of reversing a forcing process was first considered by Barioli et al. in \cite{BBF10}, who provide the following definition.

\begin{defn}[{\cite[Definition 2.5]{BBF10}}]\label{def:ZFSreversal}
Let $S$ be a zero forcing set of a graph $G$. A \emph{reversal} of $S$, denoted $\rev(S)$, is the set of last vertices of the maximal zero forcing chains of a chronological list of forces.
\end{defn}


Anderson et al.\ use this notion of reversing a forcing set in \cite{ACF21}. As discussed there, a zero forcing set $S$ can have many different reversals, any of which can be denoted by $\rev(S)$. However, if we are given a specific set of forces $\F$ of $S$, then $\F$ defines one specific reversal of $S$. As such, instead of considering the terminus of $\F$, we can instead consider a reversal of $S$; any results proven about $\rev(S)$ are then implied to be true for any possible set of forces $\F$ of $S$. For example, in \cite[Theorem 2.6]{BBF10}, Barioli et al.\ show that any reversal of a zero forcing set is also a zero forcing set; note that this is the analogous result to Proposition \ref{prop:HFSreversing} for standard zero forcing.


Knowing that reversals of zero forcing sets are also zero forcing sets, Anderson et al.\ then prove the following lemma.

\begin{lem}[{\cite[Lemma 5.2]{ACF21}}]\label{lem:andersonrev}
Let $G$ be a graph, $S$ a zero forcing set of $G$, and let $t=\pth(G;S)$. Then, for any reversal of $S$ and $i=0,\ldots, t $,
\[
S^{(t-i)} \subseteq (\rev(S) )^{[i]}.
\]
\end{lem}

We can prove an analogous lemma for hopping, since we know that the terminus of a set of hopping forces of a hopping forcing set is also a hopping forcing set by Proposition \ref{prop:HFSreversing}. To do so, we will augment a graph with extra edges according to the hopping forces performed, then apply the lemma of Anderson et al.\ to the augmented graph to prove our desired result for the original graph. Note that while we have changed the terminology used here, considering termini instead of reversals, the two lemmas are analogous, due to the reasons discussed above.

\begin{lem}\label{lem:hoppingreversal}
Let $G$ be a graph, $B$ a hopping forcing set of $G$, $\F$ a set of hopping forces of $B$, and let $t=\pth(G;\F)$. Then, for any $i=0,\ldots, t$,
\[
\F^{(t-i)} \subseteq ( \Term(\F) )^{[i]}.
\]
\end{lem}
\begin{proof}
Create the graph $G'$ by adding the edge $uv$ to $G$ for every hopping force $u\rightarrow v$ in $\mc{F}$. Notice that $B$ is then a zero forcing set of $G'$, since if $u\rightarrow v$ is a hopping force on $G$ in $\mc{F}$, then $u$ is adjacent to only blue vertices, so adding the edge $uv$ makes $u$ adjacent to exactly one white vertex in $G'$ (namely $v$), which it can force. So, by Lemma \ref{lem:andersonrev}, since $B$ is a zero forcing set of $G'$, we have $B^{(t-i)} \subseteq (\rev(B) )^{[i]}$ for every $i\in\{0,\ldots, t \}$.

Observe that by our construction of $G'$, $B^{(t-i)}$ in $G'$ is equal to $\F^{(t-i)}$ in $G$; similarly, $(\rev(B) )^{[i]}$ in $G'$ is equal to $( \Term(\F) )^{[i]}$ in $G$. As such, since $B^{(t-i)} \subseteq (\rev(B) )^{[i]}$ in $G'$ for any $i=0,\ldots, t$, we have that $\F^{(t-i)} \subseteq ( \Term(\F) )^{[i]}$ in $G$ for any $i=0,\ldots, t$, as desired.
\end{proof}

This lemma helps give us the analogous result to Theorem \ref{thm:thz*} for hopping throttling. The proof is equivalent to that from \cite{ACF21}, but using hopping forcing parameters instead of standard zero forcing ones. As such, we omit the proof here. Intuitively though, the proof works by taking a hopping forcing set $B$ and a reversal of it $\Term(\F)$ for some set of hopping forces $\F$ of $B$. We can then consider the hopping forcing set $\hat{B}=B\cup \Term(\F)$. Lemma \ref{lem:hoppingreversal} helps us show that
\[
V(G)=\hat{B}^{\left[ \left\lceil \frac{t-1}{2} \right\rceil \right]}.
\]
In this sense, using $\hat{B}$ instead of $B$ doubles the number of blue vertices at time $0$, but halves the propagation time, so the throttling number is unchanged. Applying this fact repeatedly gives the desired result. Define $\kh(G,p)= \min\{|B| \mid \pth(G;B)=p \}$.

\begin{thm}\label{thm:thh*}
For any graph $G$ of order $n$, $\thhs(G)$ is the least $k$ such that $\pth(G,k)=1$; that is, $\thhs(G)=\kh(G,1)$. Necessarily, $\kh(G,1)\geq \frac{n}{2}$.
\end{thm}

However, there is one improvement to this theorem we can make under the hopping color change rule. Suppose $B$ is a hopping forcing set of a graph $G$. The result that $\kh(G,1)\geq \frac{n}{2}$ relies on the fact that we can force at most $|B|$ vertices to be blue in every time step. But, by our logic from Theorem \ref{thm:thhkappa}, we can improve this to at most $|B|-\kappa$ vertices in each time step, for $\kappa$ the vertex connectivity of $G$. As such, the following proposition shows how we can produce a stronger lower bound for $\kh(G,1)$.

\begin{prop}\label{prop:productkappa}
For any graph $G$ of order $n$, $\thhs(G)\geq \left\lceil \frac{n+\kappa}{2} \right\rceil$.
\end{prop}
\begin{proof}
Let $B$ be a hopping forcing set of $G$. By our reasoning from the proof of Theorem \ref{thm:thhkappa}, at most $|B|-\kappa$ vertices can be forced blue in each time step. All $n$ vertices of $G$ must be blue after $\pth(G;B)$ time steps, so
\begin{eqnarray*}
|B| + \underbrace{(|B|-\kappa) + \ldots + (|B|-\kappa)}_{\pth(G;B) \text{ times}} \geq n &\iff& |B|+\pth(G;B)(|B|-\kappa) \geq n \\
&\iff & \pth(G;B) \geq \frac{n-|B|}{|B|-\kappa}.
\end{eqnarray*}
Suppose that $\pth(G;B)=1$. We then have that
\[
1 \geq \frac{n-|B|}{|B|-\kappa}
\iff |B|-\kappa \geq n-|B|
\iff 2|B|-\kappa \geq n
\iff |B| \geq \frac{n+\kappa}{2}.
\]
This gives $\thhs(G)=\kh(G,1)= \min\{|B| \mid \pth(G;B)=1 \} \geq \frac{n+\kappa}{2}$. Since $\thhs(G)$ is an integer, we have $\thhs(G)\geq \left\lceil \frac{n+\kappa}{2} \right\rceil$, as desired.
\end{proof}

For disconnected graphs, $\kappa=0$, so the bound in Proposition \ref{prop:productkappa} is equivalent to that in Theorem \ref{thm:thh*}. For connected graphs, $\kappa\geq 1$, implying the following result.

\begin{cor}
For any connected graph $G$ of order $n$, $\thhs(G)\geq \left\lceil \frac{n+1}{2} \right\rceil$.
\end{cor}

Of special attention here is when $G\iso K_n$. In this case, the only hopping forcing set is $V(K_n)$, but recall that we do not consider subsets of size $|V(G)|$ when calculating the product hopping throttling number without initial cost. Then, since $\pth(K_n;B)=\infty$ for every proper subset $B$ of $V(G)$, $\thhs(K_n)=\infty$. However, note that Proposition \ref{prop:productkappa} still holds when $G\iso K_n$, since $\left\lceil \frac{n+\kappa}{2} \right\rceil$ is finite.

As an example, we will determine the no initial cost product hopping throttling number for the path graphs. Anderson et al.\ show that, for no initial cost product throttling for standard zero forcing, paths have the minimum possible throttling number \cite{ACF21}. Under hopping throttling, paths achieve the minimum possible value of $\thhs(G)$ for connected graphs.

\begin{prop}
For any path $P_n$ with $n\geq 3$, $\thhs(P_n)=\left\lceil \frac{n+1}{2} \right\rceil = \begin{cases}
\frac{n+1}{2} &\text{if } n \text{ is odd}; \\
\frac{n+2}{2} &\text{if } n \text{ is even}.
\end{cases}$
\end{prop}

\begin{proof}
Observe that by Proposition \ref{prop:productkappa}, since $\kappa(P_n)=1$, we have that $\thhs(P_n)\geq \left\lceil \frac{n+1}{2} \right\rceil$.

Let $V(P_n)=\{1,\ldots,n \}$ with edges $\{i,i+1 \}$ for $1\leq i\leq n-1$. If $n$ is odd, color the vertices in $B=\left\{1,\ldots, \frac{n+1}{2} \right\}$ blue. Then, there are $n-\frac{n+1}{2}=\frac{n-1}{2}=\frac{n+1}{2}-1$ white vertices in $P_n$. Observe that every vertex in $\{1,\ldots,\frac{n+1}{2}-1 \}$ is only adjacent to blue vertices and so is active; as such, this set can force all of the white vertices in one time step. This gives, for odd $n$, $\thhs(P_n;B)=\frac{n+1}{2}=\left\lceil \frac{n+1}{2} \right\rceil$.

If $n$ is even, color the vertices in $B=\left\{1,\ldots, \frac{n+2}{2} \right\}$ blue. This implies that there are $n-\frac{n+2}{2}=\frac{n-2}{2}=\frac{n+2}{2}-2$ white vertices in $P_n$. Every vertex in $\{1,\ldots,\frac{n+2}{2}-1 \}$ is only adjacent to blue vertices and is thus active. So, this set can force all of the white vertices to become blue in one time step. This gives, for even $n$, $\thhs(P_n;B)=\frac{n+2}{2}=\left\lceil \frac{n+1}{2} \right\rceil$.

Altogether, for any $n\geq 3$, $\thhs(P_n)\leq \left\lceil \frac{n+1}{2} \right\rceil$. This gives $\thhs(P_n)=\left\lceil \frac{n+1}{2} \right\rceil$, as desired.
\end{proof}


\section{Concluding remarks}\label{sec:conclusion}

Theorem \ref{thm:thhkappa} shows that we can bound the hopping throttling number of a graph $G$ on $n$ vertices in terms of its vertex connectivity $\kappa$. The proof makes use of the observation that the number of active vertices in $G$ at a particular time, given some hopping forcing set $B$ of $G$, can never be greater than $|B|-\kappa$, as there must be at least $\kappa$ blue vertices at that time that are adjacent to a white vertex and thus dormant. It would be interesting, then, to determine similar lower bounds for the throttling number under other color change rules, such as the standard throttling number $\thz(G)$. Such bounds might make use of the graph's vertex connectivity or might rely on other graph parameters, but would be helpful either way to determine throttling numbers more quickly.

The proof of Lemma \ref{lem:hoppingreversal} uses a technique in which we add edges to a graph to turn a set of hopping forces into a set of zero forces. Augmenting a graph in this way then lets us use results from standard zero forcing to immediately derive analogous results for hopping forcing. What other results from zero forcing can we apply this augmented graph technique to?

Zero forcing emerged as a technique in linear algebra to help calculate the minimum rank of different graph families. Does hopping forcing, in particular the hopping forcing number $\H(G)$, also have an explicit connection to linear algebra? Notice that, using the augmented graph technique described above, we can turn a hopping forcing process on a graph $G$ into a zero forcing process on a spanning supergraph $\hat{G}$ of $G$. This gives us that $M(\hat{G})\leq \Z(\hat{G}) \leq \H(G)$, meaning that for any augmented graph of $G$ that can be obtained by the above process, $\H(G)$ bounds the maximum nullity of $\hat{G}$. Can this relationship be more rigorously defined?

Recall that, by defining the $\CCRZf$ color change rule as the union of the standard color change rule and the hopping color change rule, we can show that
\begin{equation}\label{eqn:Zfdefn}
\CCRZf(G)=\Zf(G):= \min\{\Z(G') \mid G\minor G' \}
\end{equation}
as in Theorem \ref{thm:ccrzf}. Since hopping ultimately derives from $\Zf$ forcing, can the $\Zf$ forcing number also be expressed in terms of the hopping forcing number in a similar fashion to Equation \ref{eqn:Zfdefn}?



%
%
%
%


\end{document}